\documentclass[12pt,a4paper,reqno]{amsart}
\usepackage[T1]{fontenc}
\usepackage{amsmath,amssymb,ifthen}
\usepackage{fullpage,comment}
\usepackage{graphicx,psfrag,subfigure}
\usepackage{xcolor}
\usepackage{cite}

\def\AA{\mathbb{ A}}

\def\R{{\mathbb R}}

\def\N{{\mathbb N}}

\def\BB{{\mathcal B}}
\def\HH{{\mathcal H}}
\def\MM{{\mathcal M}}

\def\PP{{\mathcal P}}
\def\RR{{\mathcal R}}
\def\SS{{\mathcal S}}
\def\LL{{\mathcal L}}
\def\TT{{\mathcal T}}

\def\UU{\mathcal U}
\def\XX{{\mathcal X}}
\def\YY{{\mathcal Y}}
\def\ZZ{{\mathcal Z}}
\def\VV{{\mathcal V}}

\def\tend{{t_{\rm end}}}
\def\uint{u^{{\rm int}}}
\def\vint{v^{{\rm int}}}

\def\norm#1#2{\|#1\|_{#2}}

\def\snorm#1#2{\vert #1\vert_{#2}}
\def\bsnorm#1#2{\vert #1\vert_{{\rm b},#2}}
\def\set#1#2{\big\{#1\,:\,#2\big\}}

\def\eps{\varepsilon}
\newcommand{\normHmeh}[3][]{#1\|#2#1\|_{H^{-1/2}(#3)}}
\newcommand{\normLtwo}[3][]{#1\|#2#1\|_{L^2(#3)}}
\newcommand{\normHeh}[3][]{#1\|#2#1\|_{H^{1/2}(#3)}}

\def\normL2#1#2{\|#1\|_{L^2(#2)}}

\newcommand{\dual}[3][]{#1\langle#2\,,\,#3#1\rangle}

\newcounter{constantsnumber}
\def\namec#1#2{%
 \ifthenelse{\equal{#1}{rel}}{C_{\rm rel}}{%
  \ifthenelse{\equal{#1}{mesh}}{C_{\rm mesh}}{%
  \ifthenelse{\equal{#1}{sz}}{C_{\rm sz}}{%
  \ifthenelse{\equal{#1}{dislocrel}}{C_{\rm dlr}}{%
  \ifthenelse{\equal{#1}{eff}}{C_{\rm eff}}{%
  \ifthenelse{\equal{#1}{main}}{C_{\rm V}}{%
  \ifthenelse{\equal{#1}{opt}}{C_{\rm opt}}{%
  \ifthenelse{\equal{#1}{normequiv}}{C_{\rm norm}}{%
  \ifthenelse{\equal{#1}{reliable}}{C_{\rm rel}}{%
  \ifthenelse{\equal{#1}{efficient}}{C_{\rm eff}}{%
  \ifthenelse{\equal{#1}{dlr}}{C_{\rm dlr}}{%
  \ifthenelse{\equal{#1}{stable}}{C_{\rm stab}}{%
  \ifthenelse{\equal{#1}{reduction}}{C_{\rm red}}{%
   \ifthenelse{\equal{#1}{unibound}}{C_{\rm hot}}{%
    \ifthenelse{\equal{#1}{hotConst}}{C_{\rm hot}}{%
   \ifthenelse{\equal{#1}{inverseK}}{C_{\rm K}}{%
  \ifthenelse{\equal{#1}{refined}}{C_{\rm ref}}{%
  \ifthenelse{\equal{#1}{estconv}}{C_{\rm est}}{%
  \ifthenelse{\equal{#1}{optimal}}{C_{\rm opt}}{%
  \ifthenelse{\equal{#1}{qo}}{C_{\rm qo}}{%
  \ifthenelse{\equal{#1}{mon}}{C_{\rm mon}}{%
  \ifthenelse{\equal{#1}{cea}}{C_{\mbox{\scriptsize C\'ea}}}{%
  \ifthenelse{\equal{#2}{newcounter}}{\refstepcounter{constantsnumber}\label{const#1}}{}C_{\ref{const#1}}}%
}}}}}}}}}}}}}}}}}}}}}}

\newcounter{contractionnumber}
\def\nameq#1#2{%
  \ifthenelse{\equal{#1}{reduction}}{q_{\rm red}}{%
  \ifthenelse{\equal{#1}{estconv}}{q_{\rm est}}{%
  \ifthenelse{\equal{#1}{cea}}{q_{\mbox{\scriptsize C\'ea}}}{%
  \ifthenelse{\equal{#2}{newcounter}}{\refstepcounter{contractionnumber}\label{contraction#1}}{}q_{\ref{contraction#1}}}%
}}}

\def\namer#1#2{%
  \ifthenelse{\equal{#1}{reduction}}{\rho_{\rm red}}{%
  \ifthenelse{\equal{#1}{estconv}}{\rho_{\rm est}}{%
  \ifthenelse{\equal{#1}{cea}}{\rho_{\mbox{\scriptsize C\'ea}}}{%
  \ifthenelse{\equal{#1}{qo}}{\rho_{\mbox{\scriptsize qo}}}{%
  \ifthenelse{\equal{#2}{newcounter}}{\refstepcounter{contractionnumber}\label{contraction#1}}{}\rho_{\ref{contraction#1}}}%
}}}}

\newtheorem{theorem}{Theorem}
\newtheorem{proposition}[theorem]{Proposition}
\newtheorem{lemma}[theorem]{Lemma}

\newtheorem{algorithm}[theorem]{Algorithm}
\newtheorem{definition}[theorem]{Definition}
\newtheorem{remark}[theorem]{Remark}

\def\T{\mathbb T}

\makeatletter
\@namedef{subjclassname@2020}{%
  \textup{2020} Mathematics Subject Classification}
\makeatother
\begin{document}%%%%%%%%%%%%%%%%%%%%%%%%%%%%%%%%%%%%%%%%%%%%%%%%%%%%%
\title%
{Inf-sup stability implies quasi-orthogonality}
\date{\today}
\thanks{Funded by the Deutsche Forschungsgemeinschaft (DFG, German Research Foundation) -- Project-ID 258734477 -- SFB 1173 as well as the Austrian Science Fund (FWF)
under the special research program Taming complexity in PDE systems (grant SFB F65).}

\subjclass[2020]{65N30, 65N50, 15A23}
\author{Michael Feischl}
\address{Institute for Analysis and Scientific Computing,
TU Wien, Wiedner Hauptstra\ss e 8-10, 1040 Vienna.}
\email{michael.feischl@tuwien.ac.at}
\begin{abstract}
We prove new optimality results for adaptive mesh refinement algorithms for non-symmetric, indefinite, and time-dependent problems by proposing a generalization of quasi-orthogonality which follows directly from the inf-sup stability of the underlying problem.
This completely removes a central technical difficulty in modern proofs of optimal convergence of adaptive mesh refinement algorithms and leads to  simple optimality proofs for the Taylor-Hood discretization of the stationary Stokes problem, a  finite-element/boundary-element discretization of an unbounded transmission problem, and an adaptive time-stepping scheme for parabolic equations. The main technical tool are new stability bounds for the $LU$-factorization of matrices together with a recently established connection between quasi-orthogonality and matrix factorization. 
\end{abstract}

\maketitle

\section{Introduction}
Adaptive mesh-refinement for finite element methods has a huge impact in the area of computational PDEs since it was initially proposed and developed for computational fluid dynamics. Several works in the engineering  literature (see, 
e.g.,~\cite{adaptivefirst,adaptiveeng} and the references therein) show the tremendous success in many practical applications in computational sciences and engineering. 
Related books from the mathematical literature, 
e.g.,~\cite{ao00,verf}
provide many examples of a~posteriori error estimators that steer adaptive 
algorithms. 

The quest for optimal adaptive algorithms started with D\"orfler's work~\cite{d1996} proposing a strategy to mark elements for refinement. The first convergence proof in~\cite{mns} and the first convergence with rates in~\cite{bdd} laid the foundation for the modern theory of rate-optimal convergence of adaptive algorithms, which started in~\cite{stevenson07, ckns} and was later summarized in~\cite{axioms}. Those papers propose what we now call the standard adaptive loop 
\begin{align*}
 \fbox{\tt solve}\longrightarrow\fbox{\tt estimate}\longrightarrow\fbox{\tt mark}\longrightarrow \fbox{\tt refine}
\end{align*}
The new idea inspired a surge in mathematical activity in this area, see e.g., the optimality proofs in
~\cite{ks,cn} for conforming methods, in~\cite{rabus10,BeMao10,ncstokes3}    
for nonconforming methods, in~\cite{LCMHJX,CR2012,HuangXu} for mixed formulations, and in~\cite{fkmp,gantumur} for integral equations. 

However, as soon as the problem at hand is outside of the symmetric, positive definite regime, the lack of orthogonality of Galerkin solutions turns out to be a major hurdle in optimality proofs.
For symmetric problems, Galerkin orthogonality implies for increasingly accurate nested Galerkin approximations $u_\ell, u_{\ell+1},\ldots$ of the exact solution $u$ that
\begin{align}\label{eq:orth}
 \norm{u-u_{\ell+1}}{}^2+\norm{u_{\ell+1}-u_\ell}{}^2 =\norm{u-u_\ell}{}^2 
\end{align}
in the energy norm $\norm{\cdot}{}$. This Pythagoras-type identity {has been used for the first convergence proofs of adaptive finite element methods in~\cite{mns, doerfler} and} is a key tool in the optimality proofs of~\cite{stevenson07,ckns}. If the norm, however, is not induced by a symmetric scalar product corresponding to the PDE,~\eqref{eq:orth} is not true in general and hence current proofs of rate-optimality do not work anymore. The work~\cite{axioms} proposed the so-called \emph{general quasi-orthogonality} in order to circumvent the lack of orthogonality for non-symmetric and indefinite problems.

This has been successfully used for non-symmetric second-order PDEs of the form
$
 -{\rm div}(A\nabla u)+b\cdot \nabla u + cu=f,
$
where unconditional optimality of the adaptive algorithm is proved in~\cite{nonsymm} (note that under some conditions on $b$ and the initial mesh, already~\cite{cn} proved optimality). 

Those proofs, however, rely heavily on the fact that the non-symmetric part $b\cdot \nabla u$ of the operator is only a compact perturbation.
In order to tackle harder problems such as the Stokes problem (the negative definite part is more than a compact perturbation), the recent work~\cite{stokesopt} uncovered an interesting connection between general quasi-orthogonality and the $LU$-factorization of matrices.

The connection can be formulated as follows: Assume that there exists a Riesz basis $B$ of the underlying Hilbert space such that the PDE $\LL u=f$ can be equivalently stated as a matrix equation
\begin{align*}
 Mx=F\quad\text{with}\quad M\in \R^{\N\times\N},\,F\in\R^\N,
\end{align*}
where $M_{vw}=\dual{\LL w}{v}$ for all $v,w\in B$, $F_v:=\dual{f}{v}$, and $u=\sum_{v\in B}x_v v$. If the matrix $M$ has an $LU$-factorization $M=LU$ for lower/upper-triangular infinite matrices $L,U\in\R^{\N\times\N}$ such that
$L,U,L^{-1},U^{-1}\colon \ell_2\to\ell_2$ are bounded operators, then  \emph{general quasi-orthogonality} holds true. This, together with other well-known properties is sufficient to show rate-optimality of the adaptive algorithm.

To exploit this connection,~\cite{stokesopt} requires an extremely technical and problem dependent construction of a suitable Riesz basis that ensures that the matrix $M$ has wavelet-type structure with exponential decay away from the diagonal. This is then used to show that the $LU$-factorization of $M$ exists and is bounded in the correct norms.

The present work removes this major hurdle by first proposing a relaxed version of \emph{general quasi-orthogonality} and then showing that this relaxed orthogonality is a priori satisfied for uniformly inf-sup stable problems. 

The proof relies on two ingredients. First, we show that, in a certain sense, even moderately unbounded $LU$-factors are still sufficient to show optimality. Second, we prove that the spectral norms of the $LU$-factors of a matrix $M\in\R^{n\times n}$ with uniformly invertible principal submatrices are bounded by $\mathcal{O}(n^{1/2-\delta})$ for some $\delta>0$. To the best knowledge of the author, this is the first result of this kind for the $LU$-factorization.  
Both results together completely eliminate the key assumption \emph{general quasi-orthogonality} from the abstract optimality proof in~\cite{axioms} and therefore drastically simplify its application to concrete problems.

\medskip

We use this new approach in three applications: 
\begin{itemize}

\item We provide a first proof of optimal convergence of a standard adaptive algorithm for the Taylor-Hood discretization of the stationary Stokes problem in 3D and vastly simplify the proof of the 2D result from~\cite{stokesopt}. Additionally, we remove the artificial mild mesh-grading condition in~\cite{stokesopt}.

 \item We show optimality for the non-symmetric finite-element/boundary-element discretization of a 3D transmission problem (transparent boundary conditions). This vastly simplifies and generalizes a recent 2D result for the same problem~\cite{fembemopt}. 
 
 \item We propose a Crank-Nicolson method with adaptive choice of time step-size for parabolic problems and prove that it converges with optimal rates under a CFL condition. To the best of the author's knowledge, this is the first optimality result for adaptive mesh-refinement for a time-dependent PDE.
\end{itemize}

We are confident that the new approach will be useful for adaptive algorithms of many other interesting problems, including non-linear problems, time-dependent problems, and also stochastic problems.

\bigskip

The remainder of this work is organized as follows: Section~\ref{sec:abstract} presents the abstract model problem and rigorously states the adaptive algorithm and the notion of optimal convergence. Section~\ref{sec:qoN} introduces the new relaxed form of quasi-orthogonality and shows the connection to rate-optimal convergence of the adaptive algorithm. Section~\ref{sec:lu} show that a mildly unbounded $LU$-factorization of a certain matrix implies the relaxed quasi-orthogonality. Section~\ref{sec:growth} proves that regular matrices satisfy the required boundedness of the $LU$-factors and concludes the proof of the abstract main result. Finally, we apply the abstract framework to the Stokes problem in Section~\ref{sec:stokes},  the finite-element/boundary-element discretization of a full space transmission problem in Section~\ref{sec:fembem}, and to a new adaptive time-stepping scheme for parabolic equations in Section~\ref{sec:time}.

\section{The abstract setting}\label{sec:abstract}
We consider real Hilbert spaces $\XX$ and $\YY$ as well as closed subspaces $\XX_\TT\subseteq \XX$, $\YY_\TT\subseteq \YY$ 
that are based on some triangulation $\TT$ of some underlying polyhedral domain $\Omega\subseteq \R^d$ and satisfy ${\rm dim}(\XX_\TT)={\rm dim}(\YY_\TT)$. We denote the set of all admissible triangulations by $\T$ (we will specify this later) and we consider a bilinear form $a(\cdot,\cdot)\colon \XX\times \YY\to\R$ that is bounded
\begin{align}\label{eq:cont}
 |a(u,v)|\leq C_a\norm{u}{\XX}\norm{v}{\YY}\quad\text{for all }u\in\XX,\,v\in\YY
\end{align}
and uniformly inf-sup stable in the sense
\begin{align}\label{eq:uniforminfsup}
\inf_{u\in\XX}\sup_{v\in\YY}\frac{a(u,v)}{\norm{u}{\XX}\norm{v}{\YY}}\geq \gamma>0\quad\text{and}\quad   \inf_{\TT\in\T}\inf_{u\in\XX_\TT}\sup_{v\in\YY_\TT}\frac{a(u,v)}{\norm{u}{\XX}\norm{v}{\YY}}\geq \gamma>0.
 \end{align}
 Moreover, we assume that for all $v\in \YY\setminus\{0\}$, there exists $u\in \XX$ with $a(u,v)\neq 0$.
This allows us to consider solutions $u\in\XX$ and $u_\TT\in\XX_\TT$ of
\begin{align}\label{eq:weak}
 a(u,v)=f(v)\quad\text{for all }v\in\YY\quad\text{and}\quad a(u_\TT,v)=f(v)\quad\text{for all }v\in\YY_\TT
\end{align}
for some $f\in\YY^\star$.
Moreover, we immediately obtain a C\'ea-type estimate of the form
\begin{align}\label{eq:cea}
 \norm{u-u_\TT}{\XX}\leq \frac{C_a}{\gamma}\min_{v\in\XX_\TT}\norm{u-v}{\XX},
\end{align}
{see~\cite[Theorem~2]{xuzi} for the proof with optimal constant.}
We assume that $\bigcup_{\TT\in\T}\XX_\TT\subseteq \XX$ is dense, which obviously is a necessary condition for convergence of any numerical method.
\subsection{Adaptive mesh refinement}
We consider an initial regular and shape regular triangulation $\TT_0$ of $\Omega$ into compact simplices $T\in\TT_0$. We assume that
$\TT_0$ partitions $\Omega$ 
such that the intersection of two elements $T\neq T'\in\TT_0$ is either a common 
face, a common node, or empty. By $\T$ we denote the set of all regular triangulations that can be generated by iterated application of newest vertex bisection to $\TT_0$ (see, e.g.,~\cite{stevenson08} for details).

We consider standard mesh refinement algorithms that are steered by an error estimator $\eta(\TT) = \eta(\TT,u_\TT,f) = \sqrt{\sum_{T\in\TT} \eta_T(\TT)^2}$ that satisfies $\eta(\TT)\approx \norm{u-u_\TT}{\XX}$.  
\begin{algorithm}\label{alg:adaptive}
	\textbf{Input: } Initial mesh $\TT_0$, parameter $0<\theta<1$.\\
	For $\ell=0,1,2,\ldots$ do:
	\begin{enumerate}
		\item Compute $u_{\ell}:=u_{\TT_\ell}$ from~\eqref{eq:weak}.
		\item Compute error estimate $\eta_T(\TT)$ for all $T\in\TT_\ell$.
		\item Find a set $\MM_\ell\subseteq \TT_\ell$ of minimal 
cardinality such that
		\begin{align}\label{eq:doerfler}
		\sum_{T\in\MM_\ell}\eta_T(\TT_\ell)^2\geq \theta 
\sum_{T\in\TT_\ell}\eta_T(\TT_\ell)^2.
		\end{align}
		\item Use newest-vertex-bisection to refine at least the 
elements in $\MM_\ell$ and to obtain a new mesh $\TT_{\ell+1}$.
	\end{enumerate}
	\textbf{Output: } Sequence of adaptively refined meshes $\TT_\ell$ and 
corresponding approximations $u_\ell\in\XX_\ell:=\XX_{\TT_{\ell}}$.
\end{algorithm}

\begin{remark}
 Newest-vertex-bisection is a method to bisect $d$-dimensional simplices such that shape regularity is conserved. In order to prove optimality of the algorithms (see below), the mesh refinement has to satisfy certain properties. Newest-vertex-bisection, and a variant of it that allows hanging nodes, are two of only a few mesh refinement strategies that are known to be optimal. Details can be found in~\cite{stevenson08,bn}.  
\end{remark}

We say that Algorithm~\ref{alg:adaptive} is rate-optimal (or just optimal), if for every possible convergence rate $s>0$ that satisfies
\begin{subequations}\label{eq:opt}
\begin{align}
{C_{\rm best}:=}\sup_{N\in\N\cup\{0\}}\inf_{\TT\in\T\atop \#\TT-\#\TT_0\leq N} \eta(\TT) (N+1)^s 
<\infty
\end{align}
the output of Algorithm~\ref{alg:adaptive} also satisfies
\begin{align}
\sup_{\ell\in\N\cup\{0\}} \eta(\TT_\ell) (\#\TT_\ell)^s 
{\leq C_{\rm opt}C_{\rm best}}<\infty
\end{align}
\end{subequations}
{for some optimality constant $C_{\rm opt}>0$.}
Note that some authors use a different notion of rate-optimality, e.g.~\cite{ckns,stevenson07}, where they replace the error estimator by a weighted sum of error and data oscillations. We follow the definition in~\cite{axioms} as it is slightly more general in cases where the error estimator does not provide a lower bound for the error. This is relevant for problems such as, e.g., the FEM-BEM coupling problem in Section~\ref{sec:fembem}.

The first proofs of optimality~\eqref{eq:opt} for the Poisson problem can be found in the breakthrough papers~\cite{stevenson07,ckns, bdd}. They all use (at least implicitly) a form of the orthogonality~\eqref{eq:orth}. In the following, we will show that a generalization of~\eqref{eq:orth} holds a priori in the present setting.

\subsection{Proof of optimal convergence}\label{sec:proof}
We follow the proof in~\cite{axioms} which states four requirements (A1)--(A4) that are sufficient in order to show optimality of a given adaptive algorithm. The main technical innovation in the present work is to show that (A3) follows essentially from the well-posedness of the problem. We collect all the remaining assumptions below and keep the numbering for consistency with~\cite{axioms}:
 There exist constants $C_{\rm red}$, $C_{\rm stab}$, $C_{\rm dlr}$,
$C_{\rm ref}\geq 1$,
and $0\leq q_{\rm red}<1$
such that all refinements $\widehat\TT\in\T$ of a triangulation $\TT\in\T$ satisfy (A1), (A2), and (A4):

\begin{enumerate}

\item[(A1)] {Stability on non-refined elements}: All subsets $\mathcal{S}\subseteq\TT\cap\widehat\TT$ of non-refined elements satisfy
\begin{align*}
\Big|\Big(\sum_{T\in\mathcal{S}}\eta_T(\widehat\TT)^2\Big)^{1/2}-\Big(\sum_{T\in\mathcal{S}}\eta_T(\TT)^2\Big)^{1/2}\Big|\leq C_{\rm stab}\,\norm{u_\TT- u_{\widehat\TT}}{\XX}.
\end{align*}
\item[(A2)] {Reduction on refined elements}: There holds
\begin{align*}
\sum_{T\in\widehat\TT\setminus\TT}\eta_T(\widehat\TT)^2\leq q_{\rm red} \sum_{T\in\TT\setminus\widehat\TT}\eta_T(\TT)^2 + C_{\rm red}\norm{u_\TT- u_{\widehat\TT}}{\XX}^2.
\end{align*}

\item[(A4)]{Discrete reliability}:
There exists a subset $\RR(\TT,\widehat\TT)\subseteq\TT$ with $\TT\backslash\widehat\TT \subseteq\RR(\TT,\widehat\TT)$ and
$|\RR(\TT,\widehat\TT)|\le C_{\rm ref}|\TT\backslash\widehat\TT|$ such that
\begin{align*}
 \norm{u_{\widehat\TT}-u_\TT}{\XX}^2
 \leq C_{\rm dlr}^2\sum_{T\in\RR(\TT,\widehat\TT)}\eta_T(\TT)^2.
\end{align*}
\end{enumerate}

Note that we removed the original general quasi-orthogonality (A3) from~\cite{axioms}. In the optimality proof below, we will replace it with a relaxed version (see~\eqref{eq:qoN} below), which turns out to be true in the present setting.
Hence, to prove the main result in Theorem~\ref{thm:opt} below, 
we first establish linear convergence of the estimator in Section~\ref{sec:linconv} below. Then, we use the remaining assumptions (A1), (A2), and (A4) to complete the optimality proof laid out in~\cite{axioms} (but inspired by~\cite{stevenson07,ckns}).

It is a standard assumption in the literature of rate-optimality that the marking parameter $\theta$ must be chosen sufficiently small. In~\cite[Remark~4.16]{axioms}, this is quantified with $0<\theta<\theta_\star:=(1+C_{\rm stab}^2C_{\rm dlr}^2)^{-1}\leq 1/2$. With this, we state the main result of this work.
\begin{theorem}\label{thm:opt}
 For a continuous~\eqref{eq:cont} and uniform inf-sup stable~\eqref{eq:uniforminfsup} problem of the form~\eqref{eq:weak}, the assumptions (A1), (A2), and (A4) imply  rate-optimality~\eqref{eq:opt} of Algorithm~\ref{alg:adaptive} {for all $0<\theta<\theta_\star$}. {The optimality constant $C_{\rm opt}$ depends only on the constants in (A1), (A2), and (A4) as well as on $C_a$, $\gamma$, $d$, and $\TT_0$.}
\end{theorem}
We postpone the proof of this theorem to Section~\ref{sec:proof2}.

\section{General quasi-orthogonality and rate-optimal convergence}\label{sec:qoN}
General quasi-orthogonality is a property of the adaptive sequence only. Therefore, we  adopt the notation 
\begin{align*}
 \eta_\ell:=\eta(\TT_\ell),\quad\XX_\ell:=\XX_{\TT_\ell},\quad\YY_\ell:=\YY_{\TT_\ell},\quad u_\ell:=u_{\TT_\ell}.
\end{align*}
The only properties we are going to use are reliability in the sense $\norm{u-u_\ell}{\XX}\leq C_{\rm rel}\eta_\ell$ for all $\ell\in\N$, quasi-monotonicity in the sense $\eta_{\ell+k}^2\leq C_{\rm mon}\eta_\ell^2$ for all $\ell,k\in\N$ (both follow from (A1), (A2), and~(A4), see Section~\ref{sec:proof2} below), and nestedness in the sense $\XX_{\ell+1}\supseteq \XX_\ell$ as well as $\YY_{\ell+1}\supseteq \YY_\ell$ for all $\ell\in\N$.

\subsection{Relaxed quasi-orthogonality}
For the problems we have in mind (see Sections~\ref{sec:stokes}--\ref{sec:time}), general quasi-orthogonality is the key estimate in order to apply the abstract optimality proof of~\cite{axioms}.
The main task is to establish a relaxed version of this quasi-orthogonality and use it to prove linear convergence of the estimator.

This relaxed quasi-orthogonality reads as follows: There exists a function $C\colon \N\to \R$ such that for all $\ell,N\in\N$ holds
\begin{align}\label{eq:qoN}
 \sum_{k=\ell}^{\ell+N} \norm{u_{k+1}-u_k}{\XX}^2\leq C(N)\norm{u-u_\ell}{\XX}^2.
\end{align}
In Section~\ref{sec:linconv} below we show that $C(N)=o(N)$ {as $N\to\infty$} is (under the additional assumptions (A1), (A2), and (A4)) sufficient for linear convergence of the error estimator. In Sections~\ref{sec:lu}--\ref{sec:growth} below, we show that uniform inf-sup stability of the problem implies~\eqref{eq:qoN} with $C(N)=\mathcal{O}(N^{1-\delta})$ for some $\delta>0$. Hence, we show that the assumption (A3) \emph{general quasi-orthogonality} from~\cite{axioms} is redundant for inf-sup stable problems.

\begin{remark}
Note that in the case where $a(\cdot,\cdot)$ is a scalar product on $\XX$, the quasi-orthogonality~\eqref{eq:qoN} follows immediately from~\eqref{eq:orth} via a telescoping sum argument (even with $C(N)\simeq 1$). 

The proofs in the remainder of this work would allow us to use an even more general version of~\eqref{eq:qoN}, i.e.,
\begin{align}\label{eq:qoN2}
 \sum_{k=\ell}^{\ell+N} \norm{u_{k+1}-u_k}{\XX}^2-\eps\eta_k^2\leq C(N)\eta_\ell^2
\end{align}
as long as $\eps>0$ is sufficiently small (see also~\cite[Assumption~(A3)]{axioms}, where the above estimate is proposed with $C(N)\simeq 1$). This might even further reduce the assumptions on the bilinear form $a(\cdot,\cdot)$. However, we do not require this generalization for the applications in Sections~\ref{sec:stokes}--\ref{sec:time} and hence proceed with~\eqref{eq:qoN} for the sake of a clearer presentation.
\end{remark}
\subsection{Linear convergence}\label{sec:linconv} 
We show that general quasi-orthogonality~\eqref{eq:qoN} implies linear convergence under some general assumptions on the estimator sequence. This is a key step in the optimality proof below. While the proof of a similar result in~\cite[Proposition~4.10]{axioms} assumes $C(N)\simeq 1$, the present case with $C(N)\to\infty$ as $N\to\infty$ requires stronger arguments. 
\begin{lemma}\label{lem:Nsum0}
 Let $(\eta_\ell)_{\ell\in\N}$ satisfy \emph{estimator reduction}
 \begin{align}\label{eq:estred}
\eta_{\ell+1}^2\leq \kappa \eta_\ell^2 + C_{\rm est}\norm{u_{\ell+1}-u_\ell}{\XX}^2  
 \end{align}
for some $0<\kappa<1$, $C_{\rm est}>0$ as well as \emph{reliability}
\begin{align}\label{eq:rel}
 \norm{u-u_\ell}{\XX}\leq C_{\rm rel}\eta_\ell
\end{align}
 for all $\ell\in\N$.
Under general quasi-orthogonality~\eqref{eq:qoN}, there holds for all $\ell,N\in\N$ that
\begin{align}\label{eq:Nsum}
  \sum_{k=\ell}^{\ell+N} \eta_k^2 \leq D(N) \eta_\ell^2
 \end{align}
 with $D(N)=1+\frac{\kappa + C_{\rm est} C(N-1)C_{\rm rel}^2}{1-\kappa}$ and $C(N)$ from~\eqref{eq:qoN}.
\end{lemma}
\begin{proof}
 %Lemma~\ref{lem:luqoN} and Lemma~\ref{lem:Ustab} prove general quasi-orthogonality~\eqref{eq:qoN} with $C_{\rm qo}(N):=C(N)\lesssim 1+\sqrt{N}$. 
 Estimator reduction and general quasi-orthogonality~\eqref{eq:qoN} imply
 \begin{align*}
  \sum_{k=\ell+1}^{\ell+N} \eta_k^2&\leq \kappa \sum_{k=\ell+1}^{\ell+N}\eta_{k-1}^2 + C_{\rm est}\sum_{k=\ell+1}^{\ell+N}\norm{u_k-u_{k-1}}{\XX}^2\\
  &\leq 
  \kappa \sum_{k=\ell+1}^{\ell+N}\eta_{k-1}^2 + C_{\rm est}C(N-1)\norm{u-u_{\ell}}{\XX}^2.
 \end{align*}
Reliability~\eqref{eq:rel} of $\eta_\ell$ shows
\begin{align*}
 (1-\kappa)\sum_{k=\ell+1}^{\ell+N} \eta_k^2\leq (\kappa + C_{\rm est} C(N-1)C_{\rm rel}^2)\eta_{\ell}^2
\end{align*}
and hence~\eqref{eq:Nsum}. % with $C(N)\leq (1-\kappa)^{-1}(1+\kappa + C_{\rm est}C_{\rm qo}(N-1)C_{\rm rel})\lesssim 1+ C_{\rm qo}(N-1)\lesssim 1+\sqrt{N}$.
This concludes the proof.
\end{proof}

\begin{lemma}\label{lem:Nsum}
 Suppose the sequence $(\eta_\ell)_{\ell\in\N}\subset\R$ satisfies~\eqref{eq:Nsum}  with a bound $D(N)>1$, $N\in\N$ such that there exists $N_0\in\N$ with
\begin{align}\label{eq:Dcond}
q_{\log}:=\log(D(N_0)) - \sum_{j=1}^{N_0} D(j)^{-1}<0.
\end{align}
If $(\eta_\ell)_{\ell\in\N}$ is additionally quasi-monotone in the sense that there exists $C_{\rm mon}>0$ such that 
\begin{align}\label{eq:qmon}
\eta_{\ell+k}^2\leq C_{\rm mon}\eta_\ell^2\quad\text{for all }\ell,k\in\N,
\end{align}
then there holds with $q:=\exp(q_{\log}/N_0)<1$ and $C:=C_{\rm mon}\exp(-q_{\log})>0$ that
\begin{align*}
 \eta_{\ell+k}^2 \leq Cq^k \eta_\ell^2\quad\text{for all }k,\ell\in\N.
\end{align*}
\end{lemma}
\begin{remark}\label{rem:cond} 
 Note that $C(N)\leq CN^{1-\delta}$ for all $N\in\N$ and some $\delta>0$ implies $D(N)\leq \widetilde C N^{1-\delta}$ and is sufficient in order to satisfy the assumption in~\eqref{eq:Dcond}. Indeed, there holds
 \begin{align*}
  \log(D(N)) - \sum_{j=1}^{N} D(j)^{-1}\leq \log(\widetilde C) + (1-\delta)\log(N) - \widetilde C^{-1}\sum_{j=1}^N j^{-1+\delta}\to -\infty
 \end{align*}
as $N\to\infty$. Even the borderline case $D(N)\lesssim N$ works as long as $\lim \sup_{N\to\infty} D(N)/N$ is sufficiently small.
\end{remark}
\begin{proof}[Proof of Lemma~\ref{lem:Nsum}]
We prove by mathematical induction on $k$ that
\begin{align}\label{eq:indhyp}
   \eta_{\ell+k}^2 \leq \Big(\prod_{j=1}^k(1-D(j)^{-1})\Big)\sum_{j=\ell}^{\ell+k}\eta_{j}^2
 \end{align}
 for all $k,\ell\in\N$. To that end, note that~\eqref{eq:indhyp} is true for all $\ell\in\N$ and $k=0$ (we interpret the empty product as 1). For the induction step, assume that~\eqref{eq:indhyp} is true for all $\ell\in\N$ and some fixed $k\in\N$.
 Then, we apply~\eqref{eq:indhyp} for $\ell+1$ and show with~\eqref{eq:Nsum}
 \begin{align*}
  \eta_{\ell+k+1}^2 &\leq \Big(\prod_{j=1}^k(1-D(j)^{-1})\Big)\sum_{j=\ell+1}^{\ell+k+1}\eta_{j}^2\\
   &=\Big(
   \prod_{j=1}^k(1-D(j)^{-1})\Big)\Big(\sum_{j=\ell}^{\ell+k+1}\eta_{j}^2-\eta_{\ell}^2\Big)\\
    &\leq \Big(\prod_{j=1}^k(1-D(j)^{-1})\Big)\Big(\sum_{j=\ell}^{\ell+k+1}\eta_{j}^2-D(k+1)^{-1}\sum_{j=\ell}^{\ell+k+1}\eta_{j}^2\Big)\\
    &\leq \Big( \prod_{j=1}^{k+1}(1-D(j)^{-1})\Big)\sum_{j=\ell}^{\ell+k+1}\eta_{j}^2.
\end{align*}
This concludes the induction and proves~\eqref{eq:indhyp}. 
A final application of~\eqref{eq:Nsum} to~\eqref{eq:indhyp} shows
\begin{align*}
 \eta_{\ell+k}^2\leq \Big(D(k)\prod_{j=1}^{k}(1-D(j)^{-1})\Big)\eta_{\ell}^2\quad\text{for all }k,\ell\in\N.
\end{align*}
% and, replacing $\ell$ with $\ell+k$, we have
% \begin{align*}
%   \eta_{\ell+k}^2\leq \Big(D(k)\prod_{j=1}^{k}(1-D(j)^{-1})\Big)\eta_{\ell}^2.
% \end{align*}
It remains to calculate the constants $C$ and $q$. To this end, we observe
\begin{align*}
 \log\Big(D(k)\prod_{j=1}^{k}(1-D(j)^{-1})\Big) &= \log(D(k)) + \sum_{j=1}^k \log(1-D(j)^{-1})\\
 &\leq 
 \log(D(k)) - \sum_{j=1}^k D(j)^{-1}.
\end{align*}
Under the assumption~\eqref{eq:Dcond} on $D(\cdot)$, we may define $q_0:=\exp(q_{\log})$ to obtain $0<q_0<1$ and $\eta_{\ell+N_0}^2\leq q_0 \eta_\ell^2$ for all $\ell\in\N$. To extend this estimate to general $k\in\N$, we note that we always find $a,b\in\N$ with $b<N_0$ such that $k=aN_0+b$. Quasi-monotonicity of $\eta_\ell$ then implies 
\begin{align*}
 \eta_{\ell+k}^2 = \eta_{\ell+aN_0+b}^2\leq q_0^a\eta_{\ell+b}^2\leq C_{\rm mon}q_0^a \eta_\ell^2\leq \frac{1}{q_0}C_{\rm mon}q_0^{k/N_0}\eta_\ell^2,
\end{align*}
where we used $a\geq k/{N_0}-1$. This concludes the proof.
\end{proof}

\section{Bounded $LU$-factorization implies general quasi-orthogonality}\label{sec:lu}
The key requirement for the optimality proof of Theorem~\ref{thm:opt} in Section~\ref{sec:proof2} below is $C(N)\lesssim N^{1-\delta}$. In this section, we show that $C(N)$ is closely related to the boundedness of the $LU$-factorization of a matrix version of~\eqref{eq:weak}. While this connection has already been established in~\cite[Section~3.1]{stokesopt}, the key aspect here is that the new quasi-orthogonality~\eqref{eq:qoN} allows us to translate~\eqref{eq:weak} into  the matrix setting by use of a simple orthogonal basis instead of the highly involved wavelet-type construction in~\cite{stokesopt}.

We consider block-matrices based on a block-structure $n_0=0<n_1<\ldots<n_N\in\N$. A block matrix $M\in\R^{n_N\times n_N}$ is organized into blocks $M(i,j)\in \R^{(n_{i+1}-n_i)\times (n_{j+1}-n_j)}$ and thus can be written as
\begin{align}\label{eq:blockdef}
 M=\begin{pmatrix} M(0,0) &M(0,1) &\ldots &M(0,N-1)\\
    M(1,0) &M(1,1) &\ldots & &\\
    \vdots & & & &\\
    M(N-1,0) & \ldots & &M(N-1,N-1)
   \end{pmatrix}.
\end{align}
By $M(:,j)$ and $M(i,:)$ we denote the $j$-th block column and the $i$-th block row and similarly by $M(0:k, j)$ and $M(i,0:k)$, we refer to the first $k+1$ entries of the $i$-th column and $j$-th row, respectively. 
We use the short hand $M[k]:=M(0:k,0:k)\in \R^{n_{k+1}\times n_{k+1}}$ to denote the restriction of $M$ to the first $k+1$ block-rows and block-columns. We also use the notation for vectors $x\in \R^{n_N}$, i.e., $x[k]=x_{1:n_{k+1}}=(x_1,\ldots,x_{n_{k+1}})\in\R^{n_{k+1}}$.

For matrices $M\in\R^{N\times N}$ and vectors $x\in\R^N$, we use the $\ell_2$-norms
\begin{align*}
 \norm{x}{\ell_2}^2:=\sum_{i=1}^{N}x_i^2\quad\text{and}\quad \norm{M}{2}:=\sup_{x\in\R^N\setminus\{0\}}\frac{\norm{Mx}{\ell_2}}{\norm{x}{\ell_2}}.
\end{align*}
(Note that we start indexing vector/matrix entries with one, but matrix blocks with zero.)
We consider normalized block-$LU$-factorizations $M=LU$, which satisfy
\begin{align*}
 L(i,j)=U(j,i)=0\quad\text{for }j>i\quad\text{and}\quad L(i,i)=I_{n_{i+1}-n_i},
\end{align*}
where $I_k\in\R^{k\times k}$ denotes the identity matrix.
 Note that linear algebra shows {that there exists a unique} block-$LU$-factorization as long as all principal minors $M[j]$, $0\leq j\leq N-1$ are regular (see, e.g.,~\cite[Theorem~13.2]{higham2002}). 

\bigskip 

We aim to reformulate the original problem~\eqref{eq:weak} in a matrix setting by use of a hierarchical basis. To that end, we restrict $\XX$ and $\YY$ to the essential subspaces. Since the goal is to prove~\eqref{eq:qoN}, it suffices to consider the spaces $\XX_\ell,\XX_{\ell+1},\ldots,\XX_{\ell+N+1}$ for given $\ell,N\in\N$.
\begin{definition}\label{def:basis}
Given $\ell,N\in\N$ and the sequence of nested spaces $\XX_{\ell},\XX_{\ell+1},\ldots,\XX_{\ell+N+1}$ from Algorithm~\ref{alg:adaptive}, we may construct a nested orthonormal basis {$\BB^\XX$} (depending on $\ell$ and $N$) by combining an orthonormal basis $\BB_0^\XX$ of $\XX_\ell$ with orthonormal bases $\BB_j^\XX$ of
\begin{align*} 
 \set{v\in \XX_{\ell+j}}{v\perp^\XX \XX_{\ell+j-1}}
\end{align*}
for all $j=1,\ldots,N+1$. This results in $\BB^\XX:=\bigcup_{j=0}^{N+1} \BB_j^\XX$.
Analogously, we may define orthonormal bases $\BB_j^\YY$, $j=0,\ldots,N+1$ and $\BB^\YY$ with respect to $\YY$.
We assume that the basis functions $w_t^\XX\in\BB^\XX$ and $w_r^\YY\in\BB^\YY$ are ordered with respect to $\BB_j^\XX$ and $\BB_j^\YY$, i.e., the functions in $\BB_j^\XX$ come before those in $\BB_{j+1}^\XX$ and so on.

We define a corresponding  block-structure $n_0=0$, $n_{j+1}:=\#\BB_j^\XX$, $j=0,\ldots, N+1$ and a block-matrix $M\in \R^{{\rm dim}(\XX_{\ell+N+1})\times {\rm dim}(\XX_{\ell+N+1})}$ by
\begin{align}\label{eq:Mdef}
	 M_{rt}:=a(w_t^\XX,w_r^\YY)\quad\text{for all }1\leq r,t\leq {\rm dim}(\XX_{\ell+N+1}). 
	\end{align}
\end{definition}	
Stability~\eqref{eq:uniforminfsup} and boundedness~\eqref{eq:cont} of~\eqref{eq:weak} translate directly to the matrix setting. To see that, let ${\rm dim}:={\rm dim}(\XX_{\ell+N+1})$. 
The orthogonality of $\BB^\XX$ and $\BB^\YY$ immediately implies for $x,y\in\R^{\rm dim}$ that $Mx\cdot y =a(\sum_{i=1}^{{\rm dim}} x_i w_i^\XX,\sum_{i=1}^{{\rm dim}} y_i w^\YY_i)\leq C_a\norm{\sum_{i=1}^{{\rm dim}} x_i w^\XX_i}{\XX}\norm{\sum_{i=1}^{{\rm dim}} y_i w^\YY_i}{\YY}=C_a\norm{x}{\ell_2}\norm{y}{\ell_2}$ and hence boundedness
\begin{subequations}\label{eq:Mnorm}
\begin{align}
 \norm{M}{2}\leq C_a.
\end{align}
Moreover,~\eqref{eq:uniforminfsup} implies for $0\leq k\leq N+1$
\begin{align*}
\sup_{y\in\R^{n_{k+1}}\atop \norm{y}{\ell_2}=1}M[k]x\cdot y = \sup_{v\in\YY_{\ell+k}\atop \norm{v}{\YY}=1}a(\sum_{i=1}^{n_{k+1}} x_i w_i^\XX,v)\geq \gamma\norm{\sum_{i=1}^{n_{k+1}} x_i w_i^\XX}{\XX} =\gamma\norm{x}{\ell_2}
\end{align*}
and hence stability
\begin{align}
 \norm{M[k]^{-1}}{2}\leq 1/\gamma\quad\text{for all }0\leq k\leq N+1.
\end{align}
\end{subequations}
Note that $M$ depends on $\ell,N\in\N$, which we omit in favor of a leaner presentation, but the norms of $M$ and $M[k]^{-1}$ are bounded independently of $\ell,N$.

Definition~\ref{def:basis} allows us to transform the discrete problems~\eqref{eq:weak} on $\TT_\ell,\ldots,\TT_{\ell+N}$ into matrix problem as follows: For $0\leq k\leq N+1$ define $\lambda(k)\in\R^{n_{k+1}}$ such that $u_{\ell+k}=\sum_{n=1}^{n_{k+1}} \lambda(k)_n w_n^\XX$. We introduce the notation
	 \begin{align*}
	 \lambda(k):=(\lambda(k)_1,\lambda(k)_{2},\ldots,\lambda(k)_{n_{k+1}})&\in\R^{n_{k+1}},\\
	 F:=(f(w_1^\YY),\ldots, f(w_{{\rm dim}}^\YY))&\in\R^{\rm dim}
	 \end{align*}
	 and observe that~\eqref{eq:weak} implies
	\begin{align}\label{eq:l2sol}
	M[k]\lambda(k)=F[k]\quad\text{\rm for all }k=0,\ldots,N+1.
	\end{align}
	 This matrix formulation will be used to derive a sufficient criterion
	for general quasi-orthogonality~\eqref{eq:qoN}.
 
 The following proof is very similar to~\cite[Lemma~3.3]{stokesopt}. However, in contrast to the earlier result, we require the bound on $C(N)$ explicitly in $\norm{U}{2}$ and hence repeat the arguments.
 \begin{lemma}\label{lem:luqoN}
 Let $\ell,N\in\N$ and consider $M$ from~\eqref{eq:Mdef}. If $M$ has a block-$LU$-factorization $M=LU$,
	then there holds general quasi-orthogonality~\eqref{eq:qoN} with 
	\begin{align*}
	 C(N)\leq {\frac{C_a^2}{\gamma^2}} \norm{U}{2}^2\sup_{{k=1,\ldots,N+1}}\norm{U^{-1}(:,k)}{2}^2,
	\end{align*}
	where $C_a$ and $\gamma$ are defined in~\eqref{eq:cont}--\eqref{eq:uniforminfsup}.
\end{lemma}
\begin{proof}
	Due to the block-triangular structure, there holds $M[k]=L[k]U[k]$ for the block-$LU$-factorization.
	{Recall that the $L$-factor of the block-$LU$-factorization is even lower triangular (due to $L(k,k)=I$).
	Hence,} there holds for all $1\leq i\leq n_{k+1}$ that
	\begin{align*}
	(L[k]U[k]\lambda(k))_i&=(M[k]\lambda(k))_i=(F[k])_i=F_i=(M\lambda(N+1))_i\\
	&=(LU\lambda(N+1))_i=(L[k](U\lambda(N+1))[k])_i.
	\end{align*}
	Since $L$ and hence also $L[k]$ is regular, this shows that $(U\lambda(N+1))[k]=U[k]\lambda(k)$.
	Moreover, when we extend $\lambda(k)$ with zeros, there holds $U[k]\lambda(k)=U\lambda(k)$ due to the block-upper triangular structure of $U$. Altogether, this proves for all $0\leq k\leq N+1$
	\begin{align}\label{eq:onlyone}
	(U\lambda(N+1))_i=(U\lambda(k))_i\quad\text{for all }1\leq i\leq n_{k+1} \quad\text{and}\quad (U\lambda(k))_i=0\quad\text{for all }i>n_{k+1}.
	\end{align}	
	Moreover, since $\BB^\XX$ is $\XX$-orthonormal, there holds
	\begin{align*}
	 \norm{u_{\ell+k+1}-u_{\ell+k}}{\XX}&=\norm{\lambda(k+1)-\lambda(k)}{\ell_2}
	= \norm{U^{-1}(U\lambda(k+1)-U\lambda(k))}{\ell_2}.
	\end{align*}
    From~\eqref{eq:onlyone}, we see that $U\lambda(k+1)-U\lambda(k)$ has non-zero entries only for $j=n_{k+1}+1,\ldots,n_{k+2}$, which corresponds to the $k+1$-th block. Therefore, we may estimate
    \begin{align*}
     \norm{u_{\ell+k+1}-u_{\ell+k}}{\XX}&\leq \norm{U^{-1}(:,k+1)}{2}\norm{U\lambda(k+1)-U\lambda(k)}{\ell_2}\\
     &=\norm{U^{-1}(:,k+1)}{2}\norm{(U\lambda(N+1))_{n_{k+1}+1:n_{k+2}}}{\ell_2}.
    \end{align*}	
	This shows
	\begin{align*}
	\sum_{k=0}^{N}\norm{u_{\ell+k+1}-u_{\ell+k}}{\XX}^2&\leq  \max_{k=0,\ldots,N}\norm{U^{-1}(:,k+1)}{2}^2\sum_{k=0}^N\norm{(U\lambda(N+1))_{n_{k+1}+1:n_{k+2}}}{\ell_2}^2\\
	&=
	\max_{k=0,\ldots,N}\norm{U^{-1}(:,k+1)}{2}^2
	\norm{(U\lambda(N+1))_{n_{1}+1:n_{{N+2}}}}{\ell_2}^2\\
	&=\max_{k=0,\ldots,N}\norm{U^{-1}(:,k+1)}{2}^2\norm{U\lambda(N+1) - U\lambda(0)}{\ell_2}^2\\
	&\leq  \norm{U}{2}^2\max_{k=0,\ldots,N}\norm{U^{-1}(:,k+1)}{2}^2\norm{\lambda(N+1) - \lambda(0)}{\ell_2}^2.
	\end{align*}
	A final application of the $\XX$-orthogonality of $\BB^\XX$ together with the discrete inf-sup stability~\eqref{eq:uniforminfsup} shows
	\begin{align*}
	 \norm{\lambda(N+1) - \lambda(0)}{\ell_2}&=
	 \norm{u_{\ell+N+1}-u_\ell}{\XX}\leq {C_a/\gamma \norm{u-u_\ell}{\XX} } 
	\end{align*}
and concludes the proof.
\end{proof}

\section{Growth of the $LU$-factors of regular matrices}\label{sec:growth}

 The goal of this section is to bound the growth of the factors $L,U\in\R^{n\times n}$ of the block-$LU$-factorization from Section~\ref{sec:lu}. This is the key step in the optimality proof of Section~\ref{sec:proof2} and the main reason why the relaxed quasi-orthogonality~\eqref{eq:qoN} follows from the inf-sup stability of~\eqref{eq:weak}. Since the result might be of independent interest, and to convey the essential argument, we first consider the standard normalized $LU$-factorization and extend the result to block-matrices in Section~\ref{sec:extension} below.
 
 \subsection{Normalized $LU$-factorization}
 In the following, we consider the normalized $LU$-factorization $M=LU$ with $L_{ii}=1$, $1\leq i\leq n$. We recall the Schatten norms $\snorm{\cdot}{p}$ for $1\leq p\leq \infty$ (see, e.g.,~\cite[Equation~IV.31]{bhatia}) defined for matrices $M\in\R^{n\times n}$ via
 \begin{align*}
  \snorm{M}{p}:=\Big(\sum_{m=1}^n \sigma_m(M)^p\Big)^{1/p},
 \end{align*}
where $\sigma_m(M)$ denote the singular values of $M$ in descending order. Note that $\snorm{M}{\infty}:=\sigma_1(M)=\norm{M}{2}$ and $\snorm{M}{2}=\norm{M}{F}:=\sqrt{\sum_{i,j=1}^n |M_{ij}|^2}$ is the Frobenius norm. Note that since $\sigma_i(A^TA)=\sigma_i(A)^2$, we have $\snorm{A^TA}{p}=\snorm{A}{2p}^2$ for all $p\in\N{\cup\{\infty\}}$. {Since $\sigma_i(A)=\sigma_i(A^T)$, we also have $\snorm{A}{p}=\snorm{A^T}{p}$ for $p\in\N{\cup\{\infty\}}$.}

Schatten norms satisfy the general H\"older inequality (see, e.g.,~\cite[Lemma~XI.9.20]{hoelder}), i.e.,  for $1/r=1/p+1/q$ and $r,p,q\in [1,\infty]$ there holds 
\begin{align}\label{eq:hoelder}
 \snorm{AB}{r}\leq \snorm{A}{p}\snorm{B}{q} \quad \text{for all }A,B\in\R^{n\times n}.
\end{align}
Note that the case $p=\infty$ and $r=q$ also follows from the well-known inequality $\sigma_m(AB)\leq \sigma_1(A)\sigma_m(B)$ for all $1\leq m\leq n$. 

We use the standard notation $M_{i:j, k:\ell}\in \R^{(j-i+1)\times (\ell-k+1)}$ to refer to sub-matrices of $M$.

\begin{remark}
Note that a quite straightforward  argument (see also Lemma~\ref{lem:blockUstab} below) shows $\snorm{U^{-1}}{\infty}\lesssim \snorm{M^{-1}}{2}$. The standard norm equivalence between spectral and Frobenius norm implies $\snorm{M^{-1}}{2}\lesssim n^{1/2}\snorm{M^{-1}}{\infty}$.
In order to improve this estimate to $n^{1/2-\delta}$ (which is essential for the optimality proof blow), we use the Schatten norms $\snorm{\cdot}{p}$ for $p>2$ and aim to leverage the improved norm equivalence $\snorm{\cdot}{p}\leq n^{1/p}\snorm{\cdot}{\infty}$.
\end{remark}

The main idea of this section is to exploit the identity $(U^{-1})_{ij} = ((M|_{j\times j})^{-1})_{ij}$ for all $1\leq i\leq j$ (see, e.g.,~\cite[Proposition~1]{LU} or~\eqref{eq:luid} below) together with a Neumann series expansion of $(M|_{j\times j})^{-1}$, {where $M|_{j\times j}\in \R^{j\times j}$ denotes the upper left $j\times j$-submatrix of $M$}. Indeed, if we find $\alpha>0$ such that $\snorm{I-\alpha M|_{j\times j}M|_{j\times j}^T}{\infty}<1$ uniformly in $j=1,\ldots,n$, there holds
 \begin{align}\label{eq:neumann}
  \begin{split}
 (M|_{j\times j})^{-1} &= \alpha M|_{j\times j}^T(I-(I-\alpha M|_{j\times j}M|_{j\times j}^T))^{-1} \\
 &=  \alpha M|_{j\times j}^T\Big(
 I+\sum_{m=1}^\infty (I-\alpha M|_{j\times j}M|_{j\times j}^T)^{m}\Big).
 \end{split}
 \end{align}
While a direct estimate of $\snorm{U^{-1}}{\infty}$ via the previous two identities seems difficult, we show in Lemma~\ref{lem:Astar2} below that a representation of~\eqref{eq:neumann} via iterated triangular truncation leads to a useful bound which then implies the result in Theorem~\ref{thm:LUbound}.

Hence, we define the triangular truncation operators $\LL,\UU\colon \R^{n\times n}\to\R^{n\times n}$ by
 \begin{align*}
  \UU(M)_{ij}:=\begin{cases} M_{ij} & i\leq j,\\
                0 & i>j
               \end{cases}
 \end{align*}
 and $\LL(M)=M-\UU(M)$. It is well known, see, e.g.,~\cite[Equation~15]{tricut}, that there holds the estimate
 \begin{align}\label{eq:tri2}
  \snorm{\UU(M)}{\infty}\leq C\log(n)\snorm{M}{\infty}
 \end{align}
for some $C>0$ (see also Lemma~\ref{lem:blocktruncinfty} below). In the following, we prove a version of this bound for Schatten norms with $p<\infty$. This can be derived from similar results in the literature, see~\cite{davies,krein2}, but the short proof of the next lemma is from~\cite{mathoverflow}.
 \begin{lemma}\label{lem:truncation}
  Given $j\in\N$, there holds
  \begin{align*}
   \snorm{\UU(M)}{2^j}\leq 2^{j-1}\snorm{M}{2^j}
  \end{align*}
for all $M\in\R^{n\times n}$.
 \end{lemma}
\begin{proof}
 Note that $\UU(L)=\LL(U)=0$ for upper-triangular $U$ and strictly lower-triangular $L$. This implies
 \begin{align*}
  \UU(M)^T \UU(M) =(\LL+\UU)( \UU(M)^T \UU(M))= \UU(\UU(M)^T M) + \LL( M^T\UU(M)).
 \end{align*}
Hence, we have
\begin{align*}
 \snorm{\UU(M)}{2p}^2 &= \snorm{\UU(M)^T\UU(M)}{p}\leq \snorm{\UU(\UU(M)^T M)}{p}+\snorm{\LL( M^T\UU(M))}{p}.
\end{align*}
Let $C_p$ denote the maximum of the operator norms of $\LL$ and $\UU$ with respect to $\snorm{\cdot}{p}$. Then, the above implies
\begin{align*}
 \snorm{\UU(M)}{2p}^2\leq  C_p\big(\snorm{\UU(M)^T M}{p}+\snorm{M^T\UU(M)}{p}\big)\leq
 2C_p \snorm{\UU(M)}{2p}\snorm{M}{2p},
\end{align*}
where we used the H\"older inequality~\eqref{eq:hoelder} as well as {the invariance of $\snorm{\cdot}{p}$ with respect to transposition} in the last step. This and the analogous estimate for $\snorm{\LL(M)}{2p}$ imply immediately that $C_{2p}\leq 2 C_p$. Moreover, $\snorm{\cdot}{2}=\norm{\cdot}{F}$ implies $C_2=1$ and thus ${C_{2^j}\leq 2^{j-1}}$ for $j\in\N$. This concludes the proof.
\end{proof}

 We consider an auxiliary quantity that will be used in the proofs below.
 Given $A\in\R^{n\times n}$ and $k\in\N$, define $A_\star^k\in\R^{n\times n}$ by
 \begin{align*}
  (A^k_\star)_{ij}:=\begin{cases} ((I_j-A|_{j\times j}A|_{j\times j}^T)^k)_{ij} & i\leq j,\\
                     0 &i>j,
                    \end{cases}
 \end{align*}
 where $I_j\in\R^{j\times j}$ denotes the identity matrix.
 
 In the following, we will use the fact that for matrices $A,B\in\R^{n\times n}$ and $1\leq i\leq j\leq n$, there holds
 \begin{align}\label{eq:tricut}
  (A\UU(B))_{ij} = \sum_{k=1}^n A_{ik}\UU(B)_{kj} =  
  \sum_{k=1}^j A_{ik}B_{kj} = (A|_{j\times j}B|_{j\times j})_{ij}.
 \end{align}

 \begin{lemma}\label{lem:Astar2}
Let $A\in\R^{n\times n}$. Then, there holds for $k>1$ that
\begin{align}\label{eq:Astarit}
 A^k_\star = A_\star^{k-1} -\UU( A\UU(A^TA_\star^{k-1}))\quad\text{with}\quad A^1_\star = {I_n}-\UU((A(\UU(A^T)))
\end{align}
as well as for all $1\leq j\leq n$ that
\begin{align}\label{eq:Astarcol}
 (A_\star^k)_{1:j, j} = (I_j-A|_{j\times j}A|_{j\times j}^T) (A_\star^{k-1})_{1:j,j}.
\end{align}
Moreover, let $M\in\R^{n\times n}$ with $A:=\sqrt{\alpha}M$ such that $\max_{1\leq j\leq n}\snorm{{I_j}-\alpha M|_{j\times j}M|_{j\times j}^T}{\infty}<1$ for some $\alpha>0$. Then, $M$ has an $LU$-factorization $M=LU$ with
\begin{align}\label{eq:Mrep}
  U^{-1} = \alpha\UU\Big( M^T\Big({I_n}+\sum_{m=1}^\infty A_\star^m\Big)\Big).
\end{align}
\end{lemma}
\begin{proof}
 Similar to~\eqref{eq:tricut}, we obtain for $1\leq i\leq j$
 \begin{align}\label{eq:Astarid}
  ({I_n}-A\UU(A^T))_{ij} = ({I_n})_{ij}-\sum_{m=1}^j A_{im} A^T_{mj} = ({I_j}-A|_{j\times j}A|_{j\times j}^T)_{ij} = (A_\star^1)_{ij},
 \end{align}
 which shows $A_\star^1={I_n}-\UU(A(\UU(A^T)))$.
Furthermore, we get for $1\leq i\leq j$
\begin{align}\label{eq:idid}
(A_\star^{k-1}-A\UU(A^TA_\star^{k-1}))_{ij}&= (A_\star^{k-1})_{ij}- \sum_{m=1}^j 
A_{im}\sum_{m'=1}^j  A_{mm'}^T(A_\star^{k-1})_{m'j}
\\
\label{eq:idid2}
&=\sum_{m'=1}^j(I_j-A|_{j\times j}A|_{j\times j}^T)_{im'}(A_\star^{k-1})_{m'j}.
\end{align}
{Note that~\eqref{eq:idid2} and
\begin{align*}
\sum_{m'=1}^j(I_j-A|_{j\times j}A|_{j\times j}^T)_{im'}(A_\star^{k-1})_{m'j}=
  \big((I_j-A|_{j\times j}A|_{j\times j}^T)(I_j-A|_{j\times j}A|_{j\times j}^T)^{k-1}\big )_{ij}=(A_\star^k)_{ij}
\end{align*}
show~\eqref{eq:Astarit}. This, together with~\eqref{eq:idid} implies~\eqref{eq:Astarcol}.}

To prove~\eqref{eq:Mrep}, we use the identity $(U^{-1})_{ij} = ((M|_{j\times j})^{-1})_{ij}$ for all $1\leq i\leq j$ which {follows from elementary linear algebra and is proved in the more general block-matrix case in~\eqref{eq:luid} below (this identity is used in a similar context in~\cite[Proposition~1]{LU}).} The assumption on $M$ allows us to use~\eqref{eq:neumann} and, similar to~\eqref{eq:tricut}, we show for $1\leq i\leq j$ that
 \begin{align*}
  U^{-1}_{ij} = (M|_{j\times j})^{-1}_{ij} &= \alpha \sum_{m'=1}^j (M|_{j\times j}^T)_{im'}\Big(({I_j})_{m'j} + \sum_{m=1}^\infty ({I_j}-\alpha M|_{j\times j}M|_{j\times j}^T)^{m}_{m'j}\Big)\\
  &= 
  \alpha \Big(M^T\Big({I_n} + \sum_{m=1}^\infty A_\star^m\Big)\Big)_{ij}.
 \end{align*}
 This can be rewritten as~\eqref{eq:Mrep} and concludes the proof.
\end{proof}

\begin{proposition}\label{prop:LUbound}
Let $C,\gamma>0$.  Then, there exist constants $C_{\rm LU}>0$ and $p>2$ such that for $n\in\N$, all matrices $M\in\R^{n\times n}$ with $\snorm{M}{\infty}\leq C$ and $\max_{1\leq j\leq n}\snorm{(M|_{j\times j})^{-1}}{\infty}\leq 1/\gamma$ have an $LU$-factorization $M=LU$  with lower/upper-triangular $L,U\in\R^{n\times n}$ such that
\begin{align*}
 \snorm{L}{\infty}+\snorm{U^{-1}}{\infty}\leq C_{\rm LU} \,n^{1/p}.
\end{align*}
\end{proposition}
\begin{proof} 
The assumptions on $M$ imply regular principal minors $M|_{j\times j}$ for $1\leq j\leq n$ and hence guarantee the existence of the $LU$-factorization $M=LU$ (see, e.g.,~\cite[Corollary~3.5.4]{lubasic}).
The first step is to show that there exist $p>2$ and $0<q<1$ depending only on $\gamma$ and $C$, such that
\begin{align}\label{eq:Astarconv}
 \snorm{A_\star^k}{\infty}\lesssim q^k n^{1/p}\snorm{A}{\infty}\quad\text{for all }k\in\N.
\end{align}
To that end, recall $A=\sqrt{\alpha}M$ from Lemma~\ref{lem:Astar2} and choose $\alpha=\gamma^2/C^4$ to ensure for $x\in \R^j$
\begin{align*}
 \norm{x-A|_{j\times j}A|_{j\times j}^T x}{\ell_2}^2 &= \norm{x}{\ell_2}^2 - 2\norm{A|_{j\times j}^T x}{\ell_2}^2 + \norm{A|_{j\times j}A|_{j\times j}^T x}{\ell_2}^2
 \leq (1-\gamma^4/C^4)\norm{x}{\ell_2}^2
\end{align*}
and therefore
\begin{align}\label{eq:contraction}
\snorm{I_j-A|_{j\times j}A|_{j\times j}^T}{\infty}\leq \sqrt{1-\gamma^4/C^4}
\quad\text{for all }1\leq j\leq n.
\end{align}
This, the identity~\eqref{eq:Astarcol}, and the definition of the Frobenius norm imply
\begin{align}\label{eq:2red}
\begin{split}
 \snorm{A_\star^k}{2}^2 &= \sum_{j=1}^n \norm{(A_\star^k)_{1:j,j}}{\ell_2}^2
 \leq \max_{1\leq j\leq n}\snorm{I_j-A|_{j\times j}A|_{j\times j}^T}{\infty}^2 \sum_{j=1}^n\norm{ (A_\star^{k-1})_{1:j, j}}{\ell_2}^2\\
 &\leq (1-\gamma^4/C^4)\snorm{A_\star^{k-1}}{2}^2.
 \end{split}
\end{align}
Together with Lemma~\ref{lem:truncation} and~\eqref{eq:hoelder}, the identity~\eqref{eq:Astarit} implies
{\begin{align}\label{eq:4stab}
\begin{split}
 \snorm{A_\star^k}{4}&\leq 
 \snorm{A_\star^{k-1}}{4} +\snorm{\UU( A\UU(A^TA_\star^{k-1}))}{4}
 \leq 
  \snorm{A_\star^{k-1}}{4} +2\snorm{A\UU(A^TA_\star^{k-1})}{4}\\
  &\leq
   \snorm{A_\star^{k-1}}{4} +2\snorm{A}{\infty}\snorm{\UU(A^TA_\star^{k-1})}{4}
   \leq
     \snorm{A_\star^{k-1}}{4} +4\snorm{A}{\infty}\snorm{A^T}{\infty}\snorm{A_\star^{k-1}}{4}\\
     &\leq 
 (1+4\snorm{A}{\infty}^{2})\snorm{A_\star^{k-1}}{4}.
 \end{split}
\end{align}}
We choose $0<t<1$. The estimates~\eqref{eq:2red}--\eqref{eq:4stab} together with $\snorm{\cdot}{\infty}\leq \snorm{\cdot}{p}$ imply
\begin{align*}
 \snorm{A_\star^k}{\infty}=\snorm{A_\star^k}{\infty}^{(1-t)+t}\leq \snorm{A_\star^k}{2}^{1-t}\snorm{A_\star^k}{4}^t\leq
 \sqrt{1-\gamma^4/C^4}^{(1-t)(k-1)} (1+4\snorm{A}{\infty}^2)^{t(k-1)}\snorm{A_\star^1}{2}^{1-t}\snorm{A_\star^1}{4}^t.
\end{align*}
Since $\snorm{A}{\infty}^2\leq \gamma^2/C^2$, we find $0<t_0<1$ sufficiently small (depending only on $\gamma$ and $C$), such that
\begin{align*}
 q:=\sqrt{1-\gamma^4/C^4}^{1-t} (1+4\snorm{A}{\infty}^2)^{t}<1
\end{align*}
for all $0< t<t_0$.
This implies
\begin{align*}
  \snorm{A_\star^k}{\infty}\leq q^{k-1}\snorm{A_\star^1}{2}^{1-t}\snorm{A_\star^1}{4}^t.
\end{align*}
 Another application of~\eqref{eq:hoelder} on~\eqref{eq:Astarit} shows together with Lemma~\ref{lem:truncation} that $\snorm{A_\star^1}{2^j}\leq 1+2^{2j-2}\snorm{A}{\infty}\snorm{{A^T}}{2^{j}}\lesssim 1+\snorm{{A^T}}{2^{j}}$ for $j=1,2$. With the standard estimate {$\snorm{A^T}{p}\leq n^{1/p}\snorm{A^T}{\infty}=n^{1/p}\snorm{A}{\infty}$} (follows directly from the definition of Schatten norms), we conclude
 \begin{align*}
  \snorm{A_\star^k}{\infty}\lesssim q^{k-1}n^{(1-t)/2+t/4}\snorm{A}{\infty}.
 \end{align*}
Since $t>0$, this concludes the proof of~\eqref{eq:Astarconv} with $p=1/((1-t)/2+t/4)>2$.
 
This, together with~\eqref{eq:tri2} and the representation~\eqref{eq:Mrep} shows
\begin{align*}
 \snorm{U^{-1}}{\infty} \lesssim  \log(n)\snorm{M}{\infty}\Big(1+\sum_{m=1}^\infty \snorm{A_\star^m}{\infty}\Big)\lesssim \log(n)n^{1/p}\snorm{M}{\infty}\snorm{A}{\infty}\sum_{m=0}^\infty q^m.
\end{align*}
With $L=MU^{-1}$, we also obtain $\snorm{L}{\infty}\leq \snorm{M}{\infty}^2\log(n)n^{1/p}\snorm{A}{\infty}$. Replacing $p$ with $2<\widetilde p <p$, we absorb the logarithmic term and conclude the proof.
\end{proof}

\begin{theorem}\label{thm:LUbound}
Let $C,\gamma>0$.  Then, there exist constants $C_{\rm LU}>0$ and $p>2$ such that for $n\in\N$, all matrices $M\in\R^{n\times n}$ with $\snorm{M}{\infty}\leq C$ and $\max_{1\leq j\leq n}\snorm{(M|_{j\times j})^{-1}}{\infty}\leq 1/\gamma$ have an $LU$-factorization $M=LU$  with lower/upper-triangular $L,U\in\R^{n\times n}$ such that
\begin{align*}
 \snorm{L}{\infty}+\snorm{U}{\infty}+\snorm{L^{-1}}{\infty}+\snorm{U^{-1}}{\infty}\leq C_{\rm LU}n^{1/p}.
\end{align*}
\end{theorem} 
\begin{proof}
Proposition~\ref{prop:LUbound} shows the bound for $U^{-1}$ and $L$.
To prove the bound for $L^{-1}$ and $U$, we define the diagonal matrix $D={\rm diag}(U)\in\R^{n\times n}$ as well as the $LU$-factors $\widetilde L,\widetilde U$ of $M^T=\widetilde L\widetilde U$. Obviously, there holds $\widetilde U = DL^T$ and $\widetilde L=U^TD^{-1}$. Moreover, we may apply Proposition~\ref{prop:LUbound} to $M^T$ instead of $M$ in order to obtain
 \begin{align*}
 \snorm{\widetilde L}{\infty}+\snorm{\widetilde U^{-1}}{\infty}\leq C_{\rm LU} n^{1/p}.
\end{align*}
Altogether, this shows
\begin{align*}
 \snorm{L^{-1}}{\infty}= \snorm{D\widetilde U^{-T}}{\infty}\leq C_{\rm LU}n^{1/p}\snorm{D}{\infty}
\end{align*}
as well as $\snorm{U}{\infty}\leq \snorm{L^{-1}}{\infty}\snorm{M}{\infty}\leq CC_{\rm LU} n^{1/p}\snorm{D}{\infty}$.

It remains to bound $D$. To that end, define $R_1:=M_{1:j-1,j}$, $R_2:= (M_{j,1:j-1})^T$ and $R_3:=M_{jj}$ to write $M|_{j\times j}$ as a $(2\times 2)$-block-matrix and compute the $(2\times 2)$-block-$LU$-factorization as
	\begin{align}\label{eq:2x2}
	 \begin{split}
	 M|_{j\times j}&=\begin{pmatrix}
	 M|_{(j-1)\times (j-1)} & R_1\\ R_2^T & R_3
	 \end{pmatrix}\\
 &	 =
	 \begin{pmatrix}
	 I_{j-1} & 0\\ R_2^TM|_{(j-1)\times (j-1)}^{-1} & 1
	 \end{pmatrix}
	 \begin{pmatrix}
	 M|_{(j-1)\times (j-1)} & R_1\\0 & R_3-R_2^TM|_{(j-1)\times (j-1)}^{-1}R_1
	 \end{pmatrix}.
	 \end{split}
	\end{align}
	Uniqueness of the normalized $LU$-factorization (further factorization of $M|_{(j-1)\times (j-1)}$ will not alter the lower-right entry of the $U$-factor) implies $U_{jj}=D_{jj}= R_3-R_2^TM|_{(j-1)\times (j-1)}^{-1}R_1$ and hence
	$|D_{jj}|\leq \snorm{M}{\infty}+\snorm{M}{\infty}^2/\gamma\leq C+C^2/\gamma$ , where we used that the norm of the sub matrices $R_1,R_2,R_3$
	is bounded by the norm of the matrix $M$ as well as $\snorm{M|_{(j-1)\times (j-1)}^{-1}}{\infty}\leq 1/\gamma$. This concludes the proof.
\end{proof}

It is an interesting question whether the result of Theorem~\ref{thm:LUbound} is sharp. Since this is not important for our further investigations, we only give an example that shows that the $LU$ factorization can be unbounded under the assumptions of Theorem~\ref{thm:LUbound}. To that end, consider the modified Hilbert matrix $M\in\R^{n\times n}$ defined by
\begin{align}\label{eq:hmatrix}
 M= {\begin{pmatrix}
     1 & 1/2 & 1/3 &\ldots&\ldots & 1/n\\
     -1/2 & 1 & 1/2 &\ldots&\ldots &1/(n-1)\\
     \vdots & &   && &\vdots\\
     \vdots & &   & &&\vdots\\
     -1/n & \ldots &\ldots & -1/3 & -1/2 & 1
    \end{pmatrix}}.
\end{align}
Obviously, there holds $(M-I)^T=-(M-I)$. This shows for all $1\leq j\leq n$ that
\begin{align*}
 M|_{j\times j} x\cdot x = \norm{x}{\ell_2}^2\quad\text{for all }x\in\R^j.
\end{align*}
Hence, $\snorm{M|_{j\times j}^{-1}}{\infty}\leq 1$ for all $1\leq j\leq n$ and it is well-known (see, e.g.,~\cite{LU}) that $\snorm{M}{\infty}\leq C$ holds uniformly in $n\in\N$.
This shows that $M$ satisfies the assumptions of Theorem~\ref{thm:LUbound}.
Moreover, with $M_{\rm abs}\in\R^{n\times n}$ defined by $(M_{\rm abs})_{ij}=|M_{ij}|$, it is also known that $\snorm{M_{\rm abs}}{\infty}\to \infty$ as $n\to\infty$.
Straightforward calculation of the $LU$-factorization of $M=LU$ shows that $\LL(L)\leq 0$ and $U\geq 0$ (entry-wise). This implies $U_{\rm abs}=U$ and $\snorm{L}{\infty}\geq \snorm{\LL(L_{\rm abs})}{\infty}-1\geq \snorm{L_{\rm abs}}{\infty}-2$. Hence, there holds
\begin{align*}
 (\snorm{L}{\infty}+2)\snorm{U}{\infty} \geq \snorm{L_{\rm abs}}{\infty}\snorm{U_{\rm abs}}{\infty}\geq 
 \snorm{L_{\rm abs}U_{\rm abs}}{\infty}\geq \snorm{M_{\rm abs}}{\infty}\to \infty\quad\text{ as }n\to \infty.
\end{align*}
Thus, at least one of the factors $L$ or $U$ must be unbounded in $\snorm{\cdot}{\infty}$ as $n\to\infty$. Numerical experiments suggest that the $L$-factor remains bounded (probably due to the normalization $L_{ii}=1$, $1\leq i\leq n$) but the $U$-factor diverges with $\snorm{U}{\infty}\gtrsim n^{0.35}$, see Figure~\ref{fig:lunorm}.

\begin{figure}
\psfrag{lfactor}{\tiny $\snorm{L}{\infty}/\snorm{M}{\infty}$}
\psfrag{ufactor}{\tiny $\snorm{U}{\infty}/\snorm{M}{\infty}$}
\psfrag{n}{\tiny $n$}
 \includegraphics[width=0.5\textwidth]{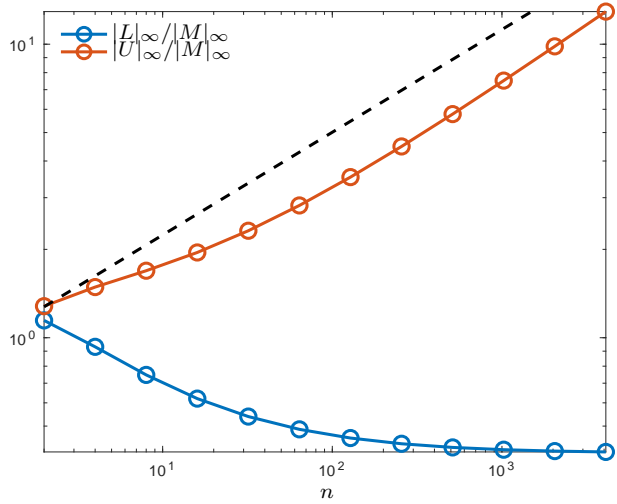}
 \caption{$LU$-factorization of the matrix $M$ from~\eqref{eq:hmatrix}. The dashed line represents $\mathcal{O}(n^{0.35})$.} \label{fig:lunorm}
\end{figure}

\subsection{Extension to the block-$LU$-factorization}\label{sec:extension}
\def\DD{\mathcal{D}}

We cannot directly use the result of Theorem~\ref{thm:LUbound} for the matrix $M$ of Definition~\ref{def:basis} since it would only give an estimate of the form $\snorm{U}{\infty}\lesssim {\rm dim}(\XX_{\ell+N+1})^{1/2-\delta}$. Since ${\rm dim}(\XX_{\ell+N+1})\gg N$ in general, this cannot be used in Lemma~\ref{lem:luqoN} to obtain general quasi-orthogonality with $C(N)=o(N)$. However, Lemma~\ref{lem:luqoN} only requires a bound on the block-$LU$-factorization of $M$ and we show below that this can be bounded by $N^{1/2-\delta}$.

We closely follow the proof of the previous section. However, the Schatten norms have no obvious generalization for block-matrices and we have to come up with custom norms. The main goal is to design norms which have all the properties used in the proofs of the previous section and are bounded by $m^{1/p}\snorm{M}{\infty}$ for block-matrices with $m\in\N$ blocks.

We assume we have a block-structure $n_0<n_1<n_2<\ldots<n_m=n$ (later we identify $m=N+2$ in order to establish the link to Section~\ref{sec:lu}). The matrix $M\in\R^{n\times n}$ in this section will always be a block-matrix with block-structure $n_0,\ldots,n_m$ (see~\eqref{eq:blockdef} for the definition of block-matrices).
We define the matrix space $\DD_{\rm b}\subset \R^{n\times m}$ by
\begin{align*}
 X\in\DD_{\rm b}\quad\Longleftrightarrow\quad X_{ij}=0\text{ for all } i\notin [n_{j-1}+1,n_j].
\end{align*}
This means that  the matrices $X\in\DD_{\rm b}$ are of the form
\begin{align*}
 X=\begin{pmatrix}
    | &  & & \boldsymbol{0}\\
     & | &  & \\
     & & \ddots &  \\
    \boldsymbol{0} & &  &|
   \end{pmatrix}\in \R^{n\times m},
\end{align*}
where $|$ symbolizes a column vector of length $n_{j-1}-n_j$ in the $j$-th column of $X$.
With this, we define the norms
\begin{align*}
 %\bsnorm{A}{p}&:=\sqrt{\sum_{k=1}^{m}\snorm{A(:,k)}{\infty}^2}
 %\quad\text{and}\\
 \bsnorm{A}{p}&:=\sup_{X\in\DD_{\rm b}\atop \snorm{X}{\infty}\leq 1}\snorm{AX}{p}
 \quad{\text{for all }A\in\R^{n\times n}\text{ and all }p \in\N\cup\{\infty\}.}
 \end{align*}
Note that $\snorm{X}{\infty}\leq 1$ for $X\in\DD_{\rm b}$ is equivalent to the fact that each column's $\norm{\cdot}{\ell_2}$-norm is bounded by one. {Moreover, $\bsnorm{\cdot}{p}$ is the operator norm for $A\colon (\DD_{\rm b},\snorm{\cdot}{\infty})\to (\R^{n\times m},\snorm{\cdot}{p})$ and thus indeed a norm.}
\begin{lemma}\label{lem:bsnorm}
 There holds for $p\in \N$ and $A,B\in\R^{n\times n}$ that
 \begin{itemize}
  \item[(i)]$\bsnorm{AB}{p}\leq \snorm{A}{\infty}\bsnorm{B}{p}$,
  \item[(ii)] $ \snorm{A}{\infty}\leq \bsnorm{A}{p}$ {and $\bsnorm{A}{\infty}=\snorm{A}{\infty}$},
  %\item[(iii)] $\bsnorm{\UU_{\rm b}(A)}{2^j}\leq 2^{j-1}\bsnorm{A}{2^j}$, $j\in\N$,
  \item[(iii)]$\bsnorm{A}{p}\leq m^{1/p}\snorm{A}{\infty}$,
  %\item[(iv)] $\bsnorm{A+B}{p}\leq \bsnorm{A}{p}+\bsnorm{B}{p}$.
  \item[(iv)] Let $C\in \R^{n\times n}$ be block-upper-triangular with $C(0:j,j) = B_jA(0:j,j)$ for $B_j\in\R^{n_{j+1}\times n_{j+1}}$ and $j=0,\ldots,m-1$. Then, there holds
  \begin{align*}
   \bsnorm{C}{2}\leq \big(\max_{j=0,\ldots,m-1} \snorm{B_j}{\infty} \big)\bsnorm{A}{2}.
  \end{align*}

 \end{itemize} 
\end{lemma}
\begin{proof}
Sub-multiplicativity~(i) is clear as $\snorm{\cdot}{p}$ satisfies the H\"older inequality~\eqref{eq:hoelder}. 
For~(ii), we note that for {$p\in\N\cup\{\infty\}$, there holds} $\bsnorm{A}{p}\geq \sup_{X\in\DD_{\rm b}\atop \snorm{X}{\infty}\leq 1}\snorm{AX}{\infty}$. Let $x\in\R^n$ with $\norm{x}{\ell_2}=1$ such that $\norm{Ax}{\ell_2}=\snorm{A}{\infty}$. Construct $X\in\DD_{\rm b}$ with $\snorm{X}{\infty}=1$ such that $x=Xy$ for 
\begin{align*}
y=(\norm{x_{1:n_1}}{\ell_2},\ldots,\norm{x_{n_{m-1}+1:n_m}}{\ell_2})\in\R^m.
\end{align*}
Note that $\norm{y}{\ell_2}=\norm{x}{\ell_2}=1$ and hence
\begin{align*}
 \bsnorm{A}{p} \geq \snorm{AX}{\infty} \geq \norm{AXy}{\ell_2}=\snorm{A}{\infty}.
\end{align*}
{For $p=\infty$, the definition of $\bsnorm{\cdot}{\infty}$ shows $\bsnorm{A}{\infty}\leq \snorm{A}{\infty}$ and hence $\bsnorm{A}{\infty}= \snorm{A}{\infty}$.}

Property~(iii) follows from the norm equivalence $\snorm{\cdot}{p}\leq \min\{m,n\}^{1/p}\snorm{\cdot}{\infty}$ in $\R^{n\times m}$, i.e.,
\begin{align*}
 \bsnorm{A}{p}\leq \min\{m,n\}^{1/p}\sup_{X\in\DD_{\rm b}\atop \snorm{X}{\infty}\leq 1}\snorm{AX}{\infty}\leq m^{1/p}\snorm{A}{\infty}.
\end{align*}
%The triangle inequality~(iv) follows directly from the triangle inequality of $\snorm{\cdot}{p}$ and the sub-additivity of the supremum. 

To see~(iv), we note that by definition of the Frobenius norm and since $X\in\DD_{\rm b}$ is block-diagonal, there holds
\begin{align*}
 \bsnorm{C}{2}^2 &=\sup_{X\in\DD_{\rm b}\atop \snorm{X}{\infty}\leq 1}\sum_{i=0}^{m-1} \norm{(CX)_{:,i+1}}{\ell_2}^2= \sup_{X\in\DD_{\rm b}\atop \snorm{X}{\infty}\leq 1}\sum_{i=0}^{m-1} \norm{C(0:i,i)X_{n_i+1:n_{i+1},i+1}}{\ell_2}^2\\
 &= \sup_{X\in\DD_{\rm b}\atop \snorm{X}{\infty}\leq 1}\sum_{i=0}^{m-1} \norm{B_iA(0:i,i)X_{n_i+1:n_{i+1},i+1}}{\ell_2}^2\\
 &\leq \big(\max_{j=0,\ldots,m-1} \snorm{B_j}{\infty} \big)^2 \sup_{X\in\DD_{\rm b}\atop \snorm{X}{\infty}\leq 1}\sum_{i=0}^{m-1} \norm{A(0:i,i)X_{n_i+1:n_{i+1},i+1}}{\ell_2}^2\\
 &\leq  \big(\max_{j=0,\ldots,m-1} \snorm{B_j}{\infty} \big)^2
 \bsnorm{A}{2}^2,
\end{align*}
where we used the block-upper-triangular structure of $C$ in the second equality. This concludes the proof.
\end{proof}

In order to mimic the proof of the previous section for block-matrices, we require a bound on the norm of the block-triangular-truncation operator $\UU_{\rm b}$ defined by
\begin{align*}
 \UU_{\rm b}(M)(i,j):=\begin{cases} M(i,j) &i\leq j,\\ 0 &i>j.\end{cases}
\end{align*}

\begin{lemma}\label{lem:blocktruncation}
  There holds for ${q}\in\{1,2\}$ that
  \begin{align*}
   \bsnorm{\UU_{\rm b}(M)}{2^q}\leq 2^{q-1}\bsnorm{M}{2^q}
  \end{align*}
for all $M\in\R^{n\times n}$.
 \end{lemma}
\begin{proof}
We define the semi-block-truncation operator $\UU_{sb}$ for $A\in\R^{n\times m}$ by
\begin{align*}
 \UU_{sb}(A)_{ij}:=\begin{cases} A_{ij} & i\leq n_{j},\\
                    0 & i> n_j.
                   \end{cases}
\end{align*}
This is just the block-truncation if we assume the block-structure only for the rows and block-columns of size one.
We have for $n_{k-1}<i\leq  n_{k}$, $k=1,\ldots,m$ and $j=1,\ldots,m$ that
 \begin{align*}
  (\UU_{\rm b}(A)X)_{ij} &=  \sum_{t=1}^{n} \UU_{\rm b}(A)_{it}X_{tj}=\sum_{t=n_{k-1}+1}^{n} A_{it}X_{tj}
  \end{align*}
  Since for $X\in\DD_{\rm b}$, we have  $X_{tj} = 0$ for $t\notin [n_{j-1}+1, n_j]$, this implies
  \begin{align*}
 (\UU_{\rm b}(A)X)_{ij} ={\left\{\begin{array}{ll}
  (AX)_{ij}, & k\leq j\\
  0, &k>j
  \end{array} \right\}}
  = \UU_{sb}(AX)_{ij}.
 \end{align*}
With this, the assertion follows immediately {for $A=M$ and} $q=1$ as
\begin{align*}
 \snorm{\UU_{\rm b}(M)X}{2}=\snorm{\UU_{sb}(MX)}{2} = \sqrt{\sum_{j=1}^{m} \sum_{i=1}^{n_{j}}(MX)_{ij}^2}\leq \snorm{MX}{2}.
\end{align*}
The case $q=2$ follows similarly to that of Lemma~\ref{lem:truncation}. 
There holds for $1\leq i\leq j\leq m$
 \begin{align*}
  \big(\UU_{sb}(A)^T \UU_{sb}(A))_{ij} &=  \sum_{t=1}^{n_i} (\UU_{sb}(A)^T)_{it} \UU_{sb}(A)_{tj}
  =\sum_{t=1}^{n_i} (\UU_{sb}(A)^T)_{it} A_{tj}\\
  &=\sum_{t=1}^{n} (\UU_{sb}(A)^T)_{it} A_{tj}
  =(\UU_{sb}(A)^T A)_{ij}.
 \end{align*}
 Similarly, we get for $j< i$ that $\big(\UU_{sb}(A)^T \UU_{sb}(A))_{ij} =(A^T \UU_{sb}(A))_{ij}$ and therefore
 \begin{align*}
 \UU_{sb}(A)^T \UU_{sb}(A)= \UU(\UU_{sb}(A)^T A)+\LL( (A)^T\UU_{sb}(A)).
 \end{align*}
Hence, we have with $MX\in\R^{n\times m}$ that
\begin{align*}
 \bsnorm{\UU_{\rm b}(M)}{4}^2 &=\sup_{X\in\DD_{\rm b}\atop \snorm{X}{\infty}\leq 1}\snorm{\UU_{\rm b}(M)X}{4}^2\\
 &= \sup_{X\in\DD_{\rm b}\atop \snorm{X}{\infty}\leq 1}\snorm{X^T\UU_{\rm b}(M)^T\UU_{\rm b}(M)X}{2} =
 \sup_{X\in\DD_{\rm b}\atop \snorm{X}{\infty}\leq 1}\snorm{\UU_{sb}(MX)^T\UU_{sb}(MX)}{2} \\
 & \leq \sup_{X\in\DD_{\rm b}\atop \snorm{X}{\infty}\leq 1}\Big(\snorm{\UU(\UU_{sb}(MX)^T MX)}{2}+\snorm{\LL( (MX)^T\UU_{sb}(MX))}{2}\Big).
\end{align*}
The above, Lemma~\ref{lem:truncation}, and~\eqref{eq:hoelder} imply
\begin{align*}
 \bsnorm{\UU_{\rm b}(M)}{4}^2&\leq  \sup_{X\in\DD_{\rm b}\atop \snorm{X}{\infty}\leq 1}\big(\snorm{\UU_{sb}(MX)^T MX}{2}+\snorm{(MX)^T\UU_{sb}(MX)}{2}\big)\leq
 2\sup_{X\in\DD_{\rm b}\atop \snorm{X}{\infty}\leq 1}\snorm{\UU_{sb}(MX)}{4}\snorm{MX}{4}\\
 &=
 2\sup_{X\in\DD_{\rm b}\atop \snorm{X}{\infty}\leq 1}\snorm{\UU_{{\rm b}}(M)X}{4}\snorm{MX}{4}\leq 2\bsnorm{\UU_{\rm b}(M)}{4}\bsnorm{M}{4}.
\end{align*}
This concludes the proof.
\end{proof}
The proof technique of the following result is well-known (see, e.g.,~\cite{log2}) and we provide the straightforward extension to block-matrices.
\begin{lemma}\label{lem:blocktruncinfty}
 There holds
 \begin{align*}
  \snorm{\UU_{\rm b}(M)}{\infty}\leq (\lceil\log_2(m)\rceil+1)\snorm{M}{\infty}.
 \end{align*}
\end{lemma}
\begin{proof}
  The idea of the proof is to write 
 \begin{align}\label{eq:log2}
 \UU_{\rm b}(M)=\sum_{i=1}^r A_i\quad\text{with } \snorm{A_i}{\infty}\leq \snorm{M}{\infty}
 \end{align}
 and $r\leq \log_2(m)+1$ whenever $m$ is a power of two. This is done by induction on $m$. Assume $m=1$, then $A_1=M$ and $r=1$, since $\UU_{\rm b}(M)=M$.
 Assume that~\eqref{eq:log2} is true for block-matrices with $m$ rows and columns. Let $M$ be a block-matrix with $2m$ rows and columns. Then, we may partition 
 \begin{align*}
  M=\begin{pmatrix} M(0:m-1,0:m-1) & M(0:m-1,m:2m-1)\\
     M(m:2m-1,0:m-1) & M(m:2m-1,m:2m-1)
    \end{pmatrix}.
 \end{align*}
The induction assumption provides $\UU_{\rm b}(M(0:m-1,0:m-1))=\sum_{i=1}^r A_{i,1}$ and $
 \UU_{\rm b}(M(m:2m-1,m:2m-1))=\sum_{i=1}^r A_{i,2}$ with $r\leq \log_2(m)+1$. We define 
 \begin{align*}
  A_i:=\begin{pmatrix} A_{i,1} & 0 \\ 0 & A_{i,2}\end{pmatrix}\quad\text{for }i=1,\ldots,r\quad\text{and}\quad A_{r+1}:=\begin{pmatrix} 0& M(0:m-1,m:2m-1) \\ 0 & 0\end{pmatrix},
 \end{align*}
where $\snorm{A_i}{\infty}\leq \max_{j=1,2}\snorm{A_{i,j}}{\infty}\leq \snorm{M}{\infty}$ for $i=1,\ldots,r$ and $\snorm{A_{r+1}}{\infty}\leq \snorm{M}{\infty}$. Hence, we may write $\UU_{\rm b}(M)=\sum_{i=1}^{r+1} A_i$, thus proving~\eqref{eq:log2} for $2m$. This concludes the induction and proves~\eqref{eq:log2} for $m$ being a power of two. If $m$ is not a power of two, we may extend $M$ with zero blocks without changing the norm of the matrix or its block-triangular truncation, resulting in~\eqref{eq:log2} with $r\leq \lceil\log_2(m)\rceil+1$. The triangle inequality concludes the proof.
\end{proof}

The combination of Lemmas~\ref{lem:bsnorm}--\ref{lem:blocktruncinfty} allows us to follow the proofs of Lemma~\ref{lem:Astar2}, Proposition~\ref{prop:LUbound}, and Theorem~\ref{thm:LUbound} by 
replacing $\snorm{\cdot}{p}$ by $\bsnorm{\cdot}{p}$, matrices by block-matrices, triangular matrices with block-triangular matrices, and $\UU$ with $\UU_{\rm b}$. All calculations on the level of matrix entries in the previous section transfer verbatim to matrix blocks. For the convenience of the reader, we summarize the main steps of the proof below the following counterpart of Theorem~\ref{thm:LUbound}.

\begin{theorem}\label{thm:blockLUbound}
Let $C,\gamma>0$.  Then, there exist constants $C_{\rm LU}>0$ and $p>2$ such that for $n\in\N$, all matrices $M\in\R^{n\times n}$ with $\snorm{M}{\infty}\leq C$ and $\max_{0\leq j\leq m-1}\snorm{(M[j])^{-1}}{\infty}\leq 1/\gamma$ have a block-$LU$-factorization $M=LU$  with block-lower/upper-triangular $L,U\in\R^{n\times n}$ such that
\begin{align*}
 \snorm{L}{\infty}+\snorm{U}{\infty}+\snorm{L^{-1}}{\infty}+\snorm{U^{-1}}{\infty}\leq C_{\rm LU}m^{1/p}.
\end{align*}
\end{theorem} 
{
\begin{proof}[Sketch of proof]
Given a block-matrix $A\in\R^{n\times n}$, we define
\begin{align*}
  (A^k_\star)(i,j):=\begin{cases} ((I_{n_j}-A[j]A[j]^T)^k)(i,j) & i\leq j,\\
                     0 &i>j,
                    \end{cases}
 \end{align*}
 and prove, analogously to Lemma~\ref{lem:Astar2}, the identity~\eqref{eq:Astarit}. The same arguments also show for all $0\leq j\leq m-1$ that
\begin{align}\label{eq:Astarcolblock}
 (A_\star^k)(0:j, j) = (I_{n_j}-A[j]A[j]^T) (A_\star^{k-1})(0:j,j).
\end{align}
We define $A:=\sqrt{\alpha}M$ and use~\eqref{eq:luid} below to obtain verbatim to Lemma~\ref{lem:Astar2} that
\begin{align}\label{eq:Mrepblock}
  U^{-1} = \alpha\UU_{\rm b}\Big( M^T\Big(I_n+\sum_{m=1}^\infty A_\star^m\Big)\Big)
\end{align}
as long as $\alpha>0$ is chosen such that
$\max_{0\leq j\leq m-1}\snorm{{I_{n_j}}-\alpha M[j]M[j]^T}{\infty}<1$.
We follow the proof of Proposition~\ref{prop:LUbound} to show that $\alpha=\gamma^2/C^2$ ensures
\begin{align}\label{eq:contractionblock}
\snorm{I_{n_j}-A[j]A[j]^T}{\infty}\leq \sqrt{1-\gamma^4/C^4}
\quad\text{for all }0\leq j\leq m-1
\end{align}
and hence~\eqref{eq:Mrepblock}.
Moreover, Lemma~\ref{lem:bsnorm}~(iv) with $B_j:=(I_{n_j}-A[j]A[j]^T)$ shows
\begin{align}\label{eq:2redblock}
\begin{split}
 \bsnorm{A_\star^k}{2}^2  \leq (1-\gamma^4/C^4)\bsnorm{A_\star^{k-1}}{2}^2.
 \end{split}
\end{align}
We follow the remainder of the proof of Proposition~\ref{prop:LUbound} to obtain
$\bsnorm{A_\star^1}{p}\lesssim \bsnorm{A^T}{p}\leq m^{1/p}\snorm{A^T}{\infty}=m^{1/p}\snorm{A}{\infty}$ (Lemma~\ref{lem:bsnorm}~(iii))
and
 \begin{align*}
  \snorm{A_\star^k}{\infty}\lesssim q^{k-1}m^{(1-t)/2+t/4}\snorm{A}{\infty}\quad\text{for some }0<q<1.
 \end{align*} 
This, Lemma~\ref{lem:blocktruncinfty}, and the representation~\eqref{eq:Mrepblock} conclude
\begin{align*}
 \snorm{L}{\infty}\lesssim \snorm{U^{-1}}{\infty} \lesssim  m^{1/p}\snorm{M}{\infty}\snorm{A}{\infty}\sum_{k=0}^\infty q^k.
\end{align*}
with $2<p< 1/((1-t)/2+t/4)$.
To obtain the bound for $\snorm{U}{\infty}$, we follow the proof of Theorem~\ref{thm:LUbound} and define the block-diagonal matrix $D\in\R^{n\times n}$ with $D(i,i)=U(i,i)$, $0\leq i\leq m-1$. With the block-$LU$-factorization $M^T=\widetilde L\widetilde U$, we obtain analogously to the proof of Theorem~\ref{thm:LUbound} that
\begin{align*}
 \snorm{U}{\infty}\lesssim \snorm{L^{-1}}{\infty}= \snorm{D\widetilde U^{-T}}{\infty}\leq C_{\rm LU}m^{1/p}\snorm{D}{\infty}.
\end{align*}
Finally, uniqueness of the normalized block-$LU$-factorization implies $U(j,j)=D(j,j)= R_3-R_2^TM[j]^{-1}R_1$, where $R_1:=M(0:j,j)$, $R_2:= (M(j,0:j))^T$ and $R_3:=M(j,j)$ for all $0\leq j\leq m-1$ (see~\eqref{eq:2x2}).
This shows	$|D(j,j)|\leq \snorm{M}{\infty}+\snorm{M}{\infty}^2/\gamma\leq C+C^2/\gamma$ and concludes the proof.
\end{proof}
}

Additionally, we require the following elementary observation.
\begin{lemma}\label{lem:blockUstab}
Let $M$ satisfy $\max_{0\leq j\leq m-1}\snorm{M[j]^{-1}}{\infty}\leq 1/\gamma$ for some $0<\gamma<1$. 
Then, $M=LU$ has a block-$LU$-factorization with
\begin{align*}
 \max_{j=0,\ldots,m-1}\snorm{U^{-1}(:,j)}{\infty}\leq 1/\gamma.
\end{align*} 
\end{lemma}
\begin{proof}
We note for $0\leq j<m$ that $M[j]=L[j]U[j]$ as well as $L^{-1}(r,j)=0$ for $r<j$ and $L^{-1}(j,j)=I$. Hence, there holds for $i\leq j$
	\begin{align}\label{eq:luid}
	M[j]^{-1}(i,j) = \sum_{r=0}^jU[j]^{-1}(i,r) L[j]^{-1}(r,j) =U[j]^{-1}(i,j)=U^{-1}(i,j),
	\end{align}
	where  the last identity follows from the fact that $U^{-1}$ is block-upper triangular. This immediately implies $\snorm{U^{-1}(:,j)}{\infty} = \snorm{M[j]^{-1}(0:j,j)}{\infty}$  and thus concludes the proof.
\end{proof} 

\subsection{Proof of Theorem~\ref{thm:opt}}\label{sec:proof2} 
Reliability~\eqref{eq:rel} of the estimator sequence follows from (A4),~\eqref{eq:cea}, and the density of $\bigcup_{\TT\in\T}\XX_\TT\subseteq \XX$ via~\cite[Lemma~3.4]{axioms} with a constant $C_{\rm rel}$ that depends only on $C_{\rm dlr}$.
{The result~\cite[Lemma~4.7]{axioms} shows that (A1)--(A2) imply  estimator reduction~\eqref{eq:estred} for some constants $\kappa$ and $C>0$ that depend only on the constants in (A1) and (A2). Moreover,~\cite[Lemma~3.5]{axioms} shows that (A1), (A2), and (A4) imply quasi-monotonicity~\eqref{eq:qmon} for some constant $C_{\rm mon}>0$ that depends only on the constants in (A1), (A2),  and (A4).}

The stability~\eqref{eq:Mnorm} for the matrix $M$ from Definition~\ref{def:basis} shows that $M$ satisfies the requirements of Theorem~\ref{thm:blockLUbound}. Thus, the block-$LU$-factorization  $M=LU$ satisfies $\norm{U}{2}\lesssim N^{1/2-\delta}$ with some uniform $\delta>0$ and {hidden constant only depending on $C_a$ and $\gamma$.} Lemma~\ref{lem:blockUstab} shows $\norm{U^{-1}(:,j)}{2}\lesssim 1$ ({with hidden constant depending only on $\gamma$}) and hence Lemma~\ref{lem:luqoN} proves general quasi-orthogonality~\eqref{eq:qoN} with $C(N)\lesssim N^{1-2\delta}$.
With this, Lemmas~\ref{lem:Nsum0}--\ref{lem:Nsum} together with Remark~\ref{rem:cond} show
\begin{align}\label{eq:linconv}
 \eta_{\ell+k}^2\leq Cq^k\eta_\ell^2
\end{align}
for all $\ell,k\in\N$ and uniform constants $0<q<1$ and $C>0$.
With this, we have all the requirements of~\cite[Lemma~4.12]{axioms} to prove the so-called \emph{optimality of D\"orfler marking}. Then, the results~\cite[Lemma~4.14]{axioms} and~\cite[Proposition~4.15]{axioms} prove rate-optimality~\eqref{eq:opt} for Algorithm~\ref{alg:adaptive}.
This concludes the proof of Theorem~\ref{thm:opt}.
\qed

\begin{remark}
Note that a key argument in the proof is $C(N)=o(N)$. Without any assumptions, the triangle inequality together with the C\'ea lemma~\eqref{eq:cea} imply $\norm{u_{k+1}-u_k}{\XX}\leq 2\norm{u-u_\ell}{\XX}$ for all $k\geq \ell$ and hence already show quasi-orthogonality~\eqref{eq:qoN} with $C(N)\simeq N$. 
Under the assumptions of Theorem~\ref{thm:opt},~\eqref{eq:linconv} implies
\begin{align*}
 \sum_{k=\ell}^N \norm{u_{k+1}-u_k}{\XX}^2\lesssim \sum_{k=\ell}^\infty \eta_k^2\lesssim  \eta_\ell^2,
\end{align*}
which is the general quasi-orthogonality~\eqref{eq:qoN2} with $C(N)\simeq 1$ and $\eps=0$. Although we only show $C(N)=o(N)$ in the proof above, we obtain $C(N)\simeq 1$ after the fact. 
\end{remark}
\section{Application: The Stokes problem}\label{sec:stokes}
\def\bu{\boldsymbol{u}}
\def\bp{\boldsymbol{p}}
Optimality of the adaptive algorithm for the Taylor-Hood discretization of the Stokes problem for $d=2$ has already been proven in~\cite{stokesopt} under a mild mesh-grading condition. We revisit this problem in order to generalize the result to $d\in\{2,3\}$, to remove the mesh condition, and to present a drastically simplified proof. 

\subsection{The stationary Stokes equation}
On a polyhedral domain $\Omega\subset \R^d$, we consider the stationary Stokes problem
\begin{align}\label{eq:stokes}
\begin{split}
 -\Delta \bu +\nabla \bp &=f\quad\text{in }\Omega,\\
 {\rm div} \bu &= 0\quad\text{in }\Omega,\\
 \bu&=0\quad\text{on }\partial\Omega,\\
\int_\Omega \bp\,dx&=0
\end{split}
 \end{align}
for given functions $f\in L^2(\Omega)$ with weak solutions $\bu\in H_0^1(\Omega)^2$ and $\bp\in L^2(\Omega)$. 
We define the space $\XX:=\YY:=H^1_0(\Omega)^2\times L^2_\star(\Omega)$.

The weak formulation of~\eqref{eq:stokes} reads: Find $(\bu,\bp)\in\XX$ such that all $(v,q)\in\XX$ satisfy
\begin{align}\label{eq:weakstokes}
a((\bu,\bp),(v,q)):= \int_\Omega \nabla \bu\cdot \nabla v\,dx - \int_\Omega \bp\,{\rm div}v\,dx -\int_\Omega q\,{\rm div}\bu\,dx = \int_\Omega fv\,dx.
\end{align}
For the purpose of discretization, we choose standard Taylor-Hood elements defined by
\begin{align*}
 \XX_\TT:=\YY_\TT:= \mathcal{S}^2_0(\TT)^2 \times \mathcal{S}^1_\star(\TT),
\end{align*}
where 
\begin{align*}
 	 \PP^p(\TT)&:=\set{ v\in L^2(\Omega)}{v|_T\text{ is polynomial of degree }\leq p,\,T\in\TT},\\
 	 \mathcal{S}^p(\TT)&:= \PP^p(\TT)\cap H^1(\Omega).
 	\end{align*}
and $\mathcal{S}^p_0(\TT):=\mathcal{S}^p(\TT)\cap H^1_0(\Omega)$ as well as $\mathcal{S}^p_\star(\TT):=\set{v\in\mathcal{S}^p(\TT)}{\int_\Omega v\,dx=0}$.
Thus, the Galerkin formulation reads: Find $(\bu_\TT,\bp_\TT)\in\XX_\TT$ such that all $(v,q)\in\XX_\TT$ satisfy
\begin{align}\label{eq:galerkin}
a((\bu_\TT,\bp_\TT),(v,q)):= \int_\Omega \nabla \bu_\TT\cdot \nabla v\,dx - \int_\Omega \bp_\TT\,{\rm div}v\,dx -\int_\Omega q\,{\rm div}\bu_\TT\,dx = \int_\Omega fv\,dx.
\end{align}

We use a locally equivalent variation proposed in~\cite{gantstokes}  of the classical error estimator proposed by {Verf\"urth~\cite[Section~4.10.3]{verf}},   
i.e., for all $T\in\TT$ define
 \begin{align*}
\eta_T(\TT)^2&:={\rm diam}(T)^2\norm{f+\Delta \bu_\TT-\nabla \bp_\TT}{L^2(T)}^2 +{\rm diam}(T)\norm{[\partial_n \bu_\TT]}{L^2(\partial T\cap\Omega)}^2\\
&\qquad\qquad + {\rm diam}(T)\norm{{\rm div}(\bu_\TT)|_T}{L^2(\partial T)}^2,
\end{align*}
where $[\cdot]$ denotes the jump across an edge (face) of $\TT$. (Note that there are also other error estimators which could be used here, e.g., those in~\cite{msv}.)
The overall estimator reads
\begin{align*}
 \eta(\TT):=\Big(\sum_{T\in\TT}\eta_T(\TT)^2\Big)^{1/2}\quad\text{for all }\TT\in\T
\end{align*}
and satisfies upper and lower error bounds, i.e.,
\begin{align}\label{eq:releff}
 C_{\rm rel}^{-1}\norm{\bu-\bu_\TT,\bp-\bp_\TT}{\XX} \leq \eta(\TT)^2\leq
 C_{\rm eff}\Big(\norm{\bu-\bu_\TT,\bp-\bp_\TT}{\XX}^2 + {\rm osc}(\TT)^2\Big)^{1/2},
\end{align}
where the data oscillation term reads ${\rm osc}(\TT)^2:= \min_{g\in \PP^0(\TT)}\sum_{T\in\TT} {\rm diam}(T)^2\norm{f-g}{L^2(T)}^2$.

To fit into the abstract framework of Section~\ref{sec:abstract}, we collect velocity and pressure in one variable, i.e., $u=(\bu,\bp)\in\XX$ and $u_\TT=(\bu_\TT,\bp_\TT)\in\XX_\TT$.
According to~\cite{brezzi-fortin} (for $d=2$) and~\cite[Theorem~3.1]{boffi} (for $d=3$), the Stokes problem $a(\cdot,\cdot)\colon\XX\times\XX\to\R$ satisfies~\eqref{eq:uniforminfsup} as long as each element $T$ of $\TT\in\T$ has at least one vertex in the interior of $\Omega$. By definition of newest-vertex bisection (see, e.g.,~\cite{stevenson08}), this is satisfied automatically if each element of
$\TT_0$ has at least one interior vertex.
\subsection{Proof of the assumptions}
The problem fits into the abstract setting of Section~\ref{sec:abstract}. All assumptions of Section~\ref{sec:proof} are verified~\cite[Lemma~3.1--3.2]{gantumur} and hence Theorem~\ref{thm:opt} implies the following result.
\begin{theorem}
 Algorithm~\ref{alg:adaptive} for the Taylor-Hood discretization of the stationary Stokes problem is rate-optimal in the sense~\eqref{eq:opt} {for all $0<\theta<\theta_\star$}.
\end{theorem}
\section{Application: Non-symmetric FEM-BEM coupling}\label{sec:fembem}
Transmission problems on unbounded domains have to be discretized with artificial boundary conditions. For general (non-convex) geometries, one of the few available methods is the coupling of finite-elements in the interior with boundary-elements for the exterior (FEM-BEM coupling). The first FEM-BEM coupling approach for such a
problem was Costabel`s symmetric coupling~\cite{costabel}. While this coupling method induces 
an operator that is symmetric, it lacks positive definiteness. 
Reformulation of the method into a positive definite one destroys the symmetry. In this section, we focus on proving optimality of the adaptive algorithm for this non-symmetric one-equation coupling or Johnson-N\'ed\'elec coupling first proposed in~\cite{johned} 
(see also~\cite{fembem} for further details).
However, optimality is also open for the \emph{symmetric} method. In principle, the methods
developed here can be used directly to prove optimality for Costabel`s symmetric coupling. 

Convergence of the adaptive algorithm for FEM-BEM coupling has been shown in~\cite{afembem,fembem}, however, optimal convergence
is only available for $d=2$ under a mild mesh-grading condition and with unusual discrete spaces (see~\cite{fembemopt}). The proof below vastly simplifies the proof from~\cite{fembemopt}, generalizes it to $d\in\{2,3\}$, and removes the mesh condition as well as the requirement for somewhat artificial discretization spaces.

\subsection{The transmission problem}
In the following, $\Omega\subseteq \R^d$, $d=2,3$ is a polygonal domain with boundary $\Gamma:=\partial\Omega$.
We denote by $H^s(\Gamma)$ the usual Sobolev spaces for $s\geq 0$. For non-integer values of $s$, we use real 
interpolation to define $H^s(\Gamma)$. 
Their dual spaces $ H^{-s}(\Gamma)$ are interpreted by extending the $L^2$-scalar product.
We consider a transmission problem of the form
\begin{align}\label{eq:transmission}
\begin{split}
 -\Delta u &= F\quad\text{in }\Omega,\\
 -\Delta u &= 0\quad\text{in }\R^d\setminus\overline{\Omega},\\
 [u]&= u_0\quad\text{on }\Gamma,\\
 [\partial_n u]&=\phi_0\quad\text{on }\Gamma,\\
 |u(x)|&={\left\{\begin{array}{ll} c\log|x| + \mathcal{O}(|x|^{-1}),& d=2,\\
          \mathcal{O}(|x|^{-1}),& d=3,
         \end{array}\right\}}\quad\text{as }|x|\to \infty,
\end{split}
 \end{align}
for {$c\in\R$ (see e.g.~\cite[Theorem~8.9]{mclean} for well-posedness)} and functions $F\in L^2(\Omega)$, $u_0\in H^{1/2}(\Gamma)$, and $\phi_0\in L^2(\Gamma)$. Here, $[\cdot]$ denotes the jump over $\Gamma$ and $\partial_n$ is the normal derivative (gradient) on $\Gamma$.

We define $\XX:=\YY:=H^1(\Omega)\times H^{-1/2}(\Gamma)$. With the Newton kernel
\begin{align*}
 G(z):=\begin{cases} -\frac{1}{2\pi}\log|z|& d=2,\\
        \frac{1}{4\pi}|z|^{-1}& d=3,
       \end{cases}\quad z\in\R^d\setminus\{0\}
\end{align*}
we define the integral operators
\begin{align}\label{eq:integralops}
(V\phi)(x):=\int_\Gamma G|x-y|\phi(y)\, ds_y \quad \text{ and } \quad 
(K g)(x):=\int_\Gamma \partial_{n(y)}G|x-y|g(y)\, ds_y 
\end{align}
for all $x\in\Gamma$, where $ds_y$ denotes the surface measure. We consider a weak form of the problem above first proposed in~\cite{johned}, i.e., the Johnson-N\'ed\'elec one equation coupling: Find $(\uint,\phi)\in\XX$ such that
\begin{align}\label{eq:weakfembem}
 a((\uint,\phi),(\vint,\psi))= f(\vint,\psi)\quad\text{for all }(\vint,\psi)\in \XX
\end{align}
with
\begin{align*}
  a((\uint,\phi),(\vint,\psi)):= \dual{\nabla \uint}{\nabla \vint}_\Omega -\dual{\phi}{\vint}_\Gamma + \dual{(1/2-K)\uint}{\psi}_\Gamma +\dual{V\phi}{\psi}_\Gamma
\end{align*}
and
\begin{align*}
  f(\vint,\psi):= \dual{F}{\vint}_\Omega +\dual{\phi_0}{\vint}_\Gamma + \dual{\psi}{(1/2-K)u_0}_\Gamma.
\end{align*}
The connection to the transmission problem~\eqref{eq:transmission} is given by
\begin{align*}
 u|_\Omega = u^{\rm int}, \quad \partial_n u|_{\R^d\setminus\Omega} = -\phi,\quad\text{and } u|_{\R^d\setminus\Omega} = V\phi + K u^{\rm int}.
\end{align*}
Existence of unique solutions of the above method was first proved in~\cite{johned} for the case of smooth $\Gamma$. 
Almost three decades later, Sayas~\cite{sayas} proved existence of unique solutions also for the case of polygonal boundaries $\Gamma$. This work was extended in~\cite{fembem} to
nonlinear material parameters and other coupling methods.

In order to discretize the problem, we introduce the space of $\TT$-elementwise linear functions that are globally continuous $\SS^1(\TT)$ as well as the $\TT$-elementwise constant functions on the boundary $\PP^0(\TT|_\Gamma)$.
Given a Galerkin solution $(\uint_\TT,\phi_\TT)\in \XX_\TT:=\YY_\TT:=\mathcal{S}^{1}(\TT)\times \PP^0(\TT|_\Gamma)$ such that
\begin{align}\label{eq:solutions}
 a((\uint_\TT,\phi_\TT),(\vint,\psi))= f(\vint,\psi)\quad\text{for all }(\vint,\psi)\in \XX_\TT,
\end{align}
the corresponding residual-based error estimator (see e.g.~\cite{afp,fembem} for the derivation) reads element-wise for all $T\in\TT$
 \begin{align*}
\eta_T(\TT)^2&:={\rm diam}(T)^2\norm{F}{L^2(T)}^2 +{\rm diam}(T)\norm{[\partial_n \uint_\TT]}{L^2(\partial T\cap\Omega)}^2\\
 &\qquad + {\rm diam}(T)\norm{\phi_0+\phi_\TT-\partial_n \uint_\TT}{L^2(\partial T\cap\Gamma)}^2\\
 &\qquad+{\rm diam}(T)\norm{\partial_\Gamma((\tfrac12-K)(u_0-\uint_\TT)-V\phi_\TT)}{L^2(\partial T\cap\Gamma)}^2,
\end{align*}
where $\partial_\Gamma$ denotes the tangential derivative (gradient) on $\Gamma$.
Note that the exterior problem affects the estimator only on elements $T\in \TT$ with $T\cap\Gamma\neq \emptyset$.  
The overall estimator reads
\begin{align*}
 \eta(\TT):=\Big(\sum_{T\in\TT}\eta_T(\TT)^2\Big)^{1/2}\quad\text{for all }\TT\in\T
\end{align*}
and is reliable in the sense
\begin{align*}
 \norm{\uint-\uint_\TT}{H^1(\Omega)}+ \norm{\phi-\phi_\TT}{H^{-1/2}(\Gamma)}\leq C_{\rm rel}\eta(\TT),
\end{align*}
but is not known to also provide a lower error bound (although this is observed in practice).
We collect the functions via the notation $u_\TT:=(\uint_\TT,\phi_\TT)$ as well as $v:=(\vint,\psi)$ to fit into the abstract setting of Section~\ref{sec:abstract}.

\subsection{Proof of the assumptions}
\begin{proof}[Proof of inf-sup stability]
For~\eqref{eq:uniforminfsup}, we mention~\cite{sayas}, which gave the first stability proof for Lipschitz domains (earlier proofs used that $K$ is a compact operator on smooth domains). This is extended in~\cite{fembem} to nonlinear coefficients and other coupling methods.
% The work~\cite{fembem} shows furthermore, that there exists a stabilization $\widetilde a(\cdot,\cdot)$ and $\widetilde f(\cdot,\cdot)$ that is invariant in the sense that the solutions of~\eqref{eq:solutions} satisfy
% \begin{align*}
%  \widetilde a((\uint_\TT,\phi_\TT),(\vint,\psi))= \widetilde f(\vint,\psi)\quad\text{for all }(\vint,\psi)\in \XX_\TT
% \end{align*}
% and, moreover,
% \begin{align}\label{eq:FBell}
%  \widetilde a((\vint,\psi),(\vint,\psi))\geq \gamma\Big(\norm{\vint}{H^1(\Omega)}^2+\norm{\psi}{H^{-1/2}(\Gamma)}^2\Big)\quad\text{for all }(\vint,\psi)\in\XX
% \end{align}
% for some uniform $\gamma>0$.
% Thus, we apply the framework of the previous sections to $\widetilde a(\cdot,\cdot)$ instead of $a(\cdot,\cdot)$. This allows us to choose the operator $T$ from~\eqref{eq:infsup} as the identity operator and~\eqref{eq:FBell} implies~\eqref{eq:infsup}.
\end{proof}

The proofs of the properties (A1), (A2), and (A4) are combinations of techniques from the FEM case and from the BEM case (mainly from~\cite{fkmp}). While no expert will be surprised by the following
results, they cannot be found in the literature and we included them for completeness.
\begin{proof}[Proof of (A1)--(A2)]
 The statements~(i) and~(ii) are part of the proof of~\cite[Theorem~25]{fembem} and follow from the triangle inequality and 
 local inverse estimates for the non-local operators $ V$ and $ K$ from~\cite{invest}. The constants $C_{\rm stab},C_{\rm red},q_{\rm red}$ depend only on $\Gamma$ and the shape regularity of $\TT$
 and $\widehat\TT$.
\end{proof}

\begin{proof}[Proof of (A4)]
 The proof is essentially the combination of the corresponding proofs for FEM in~\cite{stevenson07,ckns} and BEM in~\cite{fkmp}. We define the patch $\omega(\SS,\TT):=\set{T\in\TT}{T\cap \bigcup \SS\neq \emptyset}$ for all $\SS\subseteq \TT$. For
  $v_{\widehat{\TT}}\in\mathcal{S}^{1}(\widehat\TT)\times \PP^0(\widehat\TT|_\Gamma)$, Galerkin orthogonality implies
 \begin{align*}
 a(u_{\widehat{\TT}}-u_{{\TT}},v_{\widehat{\TT}})
  &=f(v_{\widehat{\TT}}-v_{{\TT}})- a(u_{{\TT}},v_{\widehat{\TT}}-v_{{\TT}})\quad\text{for all }v_{{\TT}}\in\XX_\TT.
 \end{align*}
Recall the Scott-Zhang projection $J_\TT: H^1(\Omega)\to \mathcal{S}^{1}(\TT)$ from~\cite{scottzhang} as well as the $L^2(\Gamma)$-orthogonal projection 
$\Pi_\TT:L^2(\Gamma) \to \PP^{0}(\TT|_\Gamma)$. With this, define
\begin{align*}
 v_{{\TT}}:=(J_\TT\uint_{\widehat{\TT}},\Pi_\TT \psi_{\widehat{\TT}})\in\mathcal{S}^{1}(\TT)\times \PP^0(\TT|_\Gamma).
\end{align*}
This implies
\begin{align}\label{eq:johneddrelstart}
\begin{split}
  a(u_{\widehat{\TT}}-u_{{\TT}},v_{\widehat{\TT}})&= \dual{F}{(1-J_\TT)\uint_{\widehat{\TT}}}_{L^2(\Omega)}-\dual{\nabla \uint_\TT}{\nabla (1-J_\TT)\uint_{\widehat{\TT}}}_{L^2(\Omega)}\\
  &\qquad  +  \dual{\phi_0+\phi_\TT}{(1-J_\TT)\uint_{\widehat{\TT}}}_{L^2(\Gamma)}\\
  &\qquad + \dual{(1/2- K)(u_0-\uint_\TT)- V\phi_\TT}{(1-\Pi_\TT)\psi_{\widehat{\TT}}}_{L^2(\Gamma)}.
  \end{split}
\end{align}
$\TT$-piecewise integration by parts shows
\begin{align*}
 \dual{&F}{(1-J_\TT)\uint_{\widehat{\TT}}}_{L^2(\Omega)}-\dual{\nabla \uint_\TT}{\nabla (1-J_\TT)\uint_{\widehat{\TT}}}_{L^2(\Omega)}+\dual{\phi_0+\phi_\TT}{(1-J_\TT)\uint_{\widehat{\TT}}}_{L^2(\Gamma)}\\
 &\lesssim\sum_{T\in\TT}\normLtwo{F+\Delta\uint_\TT}{T}\normLtwo{(1-J_\TT)\uint_{\widehat{\TT}}}{T}\\
 &\qquad+\sum_{T\in\TT}\Big(\normLtwo{[\partial_n\uint_\TT]}{\partial T\cap\Omega}
 + \normLtwo{\phi_0+\phi_\TT-\partial_n\uint_\TT }{\partial T\cap\Gamma}\Big)\normHeh{(1-J_\TT)\uint_{\widehat{\TT}}}{T}.
\end{align*}
All $T\in\TT$ with $T\notin \omega(\TT\setminus\widehat\TT,\TT)$ satisfy $((1-J_\TT)\uint_{\widehat{\TT}})|_T = 0$ by locality of the Scott-Zhang projection. This and the first-order approximation properties of $J_\TT$ imply
\begin{align}\label{eq:johnedhelp}
 \dual{&F}{(1-J_\TT)\uint_{\widehat{\TT}}}_{L^2(\Omega)}-\dual{\nabla \uint_\TT}{\nabla (1-J_\TT)\uint_{\widehat{\TT}}}_{L^2(\Omega)}+\dual{\phi_0-\phi_\TT}{(1-J_\TT)\uint_{\widehat{\TT}}}_{L^2(\Gamma)}\nonumber\\
 &\lesssim\sum_{T\in\omega(\TT\setminus\widehat\TT,\TT)}\Big({\rm diam}(T)\normLtwo{F+\Delta\uint_\TT}{T}+{\rm diam}(T)^{1/2} \normLtwo{[\partial_n\uint_\TT]}{\partial T\cap\Omega}\nonumber\\
 &\qquad+ {\rm diam}(T)^{1/2}\normLtwo{\phi_0+\phi_\TT-\partial_n\uint_\TT}{\partial T\cap\Gamma}\Big)\normLtwo{\nabla\uint_{\widehat{\TT}}}{T},
\end{align}
where the hidden constant depends only on the shape regularity of $\TT$ and $\Omega$.
Consider a partition of unity of $\Gamma$ in the sense
\begin{align*}
 \sum_{z\in\Gamma\atop z\text{ node of }\TT} \xi_z = 1\quad\text{on }\Gamma
\end{align*}
with the linear nodal hat functions $\xi_z\in\SS^1(\TT|_\Gamma)$.
Since $(1-\Pi_\TT)\psi_{\widehat{\TT}} = 0 $ on $\TT\cap\widehat\TT$, the last term on the right-hand side of~\eqref{eq:johneddrelstart} satisfies
\begin{align*}
  \dual{(1/2- K)(u_0-\uint_\TT)&- V\phi_\TT}{(1-\Pi_\TT)\psi_{\widehat{\TT}}}_{L^2(\Gamma)}\\
  &=
   \dual{\sum_{z\in\bigcup(\TT\setminus\widehat\TT)\cap\Gamma\atop z\text{ node of }\TT}\xi_z\big((1/2- K)(u_0-\uint_\TT)- V\phi_\TT\big)}{(1-\Pi_\TT)\psi_{\widehat{\TT}}}_{L^2(\Gamma)}.
\end{align*}
Setting $\vint=0$ and $\psi=1$ on $T$ and zero elsewhere in~\eqref{eq:solutions} shows $ \dual{1}{(1/2- K)(u_0-\uint_\TT)- V\phi_\TT}_{L^2(T\cap\Gamma)}=0$ for all $T\in\TT$. This allows us to follow the arguments of the proof of~\cite[Proposition~5.3]{fkmp} resp.~\cite[Proposition~4]{ffkmp:part1}. We obtain
\begin{align}\label{eq:johnedhelp2}
 \dual{(1/2&- K)(u_0-\uint_\TT)- V\phi_\TT}{(1-\Pi_\TT)\psi_{\widehat{\TT}}}_{L^2(\Gamma)}\\\nonumber
 &\lesssim \Big(\sum_{T\in\omega(\TT\setminus\widehat\TT,\TT)}{\rm diam}(T)^{1/2}\normLtwo{\nabla_\Gamma\big((1/2- K)(u_0-\uint_\TT)- V\phi_\TT\big)}{T\cap\Gamma}\Big)\normHmeh{\psi_{\widehat{\TT}}}{\Gamma}.
\end{align}
The combination of~\eqref{eq:johnedhelp}--\eqref{eq:johnedhelp2} with~\eqref{eq:johneddrelstart} and the inf-sup condition from~\cite{sayas,fembem} concludes the proof of the discrete reliability~(A4) with $\RR(\TT,\widehat\TT):=\omega(\TT\setminus\widehat\TT,\TT)$
and $C_{\rm ref}$ depending only on shape regularity. 
\end{proof}

Thus, the problem fits into the abstract setting of Section~\ref{sec:abstract}. All assumptions of Section~\ref{sec:proof} are verified above and hence Theorem~\ref{thm:opt} implies the following result.
\begin{theorem}
 Algorithm~\ref{alg:adaptive} for the non-symmetric FEM-BEM discretization of the Poisson transmission problem is rate-optimal in the sense~\eqref{eq:opt} {for all $0<\theta<\theta_\star$}.
\end{theorem}

\section{Application: Adaptive Time-stepping}\label{sec:time}
We apply the adaptive Algorithm~\ref{alg:adaptive} to a time dependent parabolic problem. We discretize the problem with the classical Crank-Nicolson scheme and choose the time steps adaptively. Note that we only consider adaptivity in time and deal with a fixed discretization in space.

The difference to classical adaptive time-stepping (see, e.g.,~\cite{soderlind,claes1} for examples) is that we do not choose the time steps based on local (or past) information, but incorporate the information on the whole time interval. This requires multiple passes through the domain with decreasing step-size but allows us to prove optimality. Linear convergence~\eqref{eq:linconv} ensures that the overhead produced by this remains bounded by the cost of the last iteration of the algorithm (see, e.g.,~\cite[Theorem~8]{optcost}).

While the result in Theorem~\ref{thm:timeopt} still requires a non-optimal CFL condition to hold, it seems to be the first optimality result for an adaptive time-stepping scheme.

\subsection{Model Problem}
Let $\AA\colon \VV\to \VV^\star$ denote a coercive (i.e., $\dual{\AA v}{v}\gtrsim \norm{v}{\VV}^2$ for all $v\in\VV$) and bounded operator on some finite dimensional Hilbert space $\VV$. For $\tend>0$, we consider
the parabolic equation
\begin{align}\label{eq:parabolic}
\begin{split}
 \partial_t u +\AA u &= f\quad\text{in } [0,\tend],\\
 u(0)&=u_0.
 \end{split}
\end{align}
The initial condition satisfies $u_0\in\HH$, where $\VV\subseteq \HH\subseteq \VV^\star$ is a Gelfand triple (in the present finite dimensional setting, $\HH$ and $\VV$ contain the same elements but are equipped with different norms, e.g., $L^2$ and $H^1$).
The equation is understood in the weak sense (see, e.g.,~\cite[Section~7.1]{evansPDE}), i.e., find $u\in \XX:=L^2(0,\tend;\VV)\cap H^1(0,\tend;\VV^\star)$ such that  all $(v,\phi)\in \YY:=L^2(0,\tend;\VV)\times \HH$
satisfy
\begin{align}\label{eq:weakparabolic}
 \int_0^\tend\dual{\partial_t u}{ v} + \dual{\AA u}{v}\,dt + \dual{u(0)}{\phi}= \int_0^\tend \dual{f}{v}\,dt+\dual{u_0}{\phi},
\end{align}
where $\dual{\cdot}{\cdot}$ is the duality bracket between $\VV$ and $\VV^\star$ and also the $\HH$-scalar product.
Here, $L^2(0,\tend;\ZZ)$ and $H^1(0,\tend;\ZZ)$ denote the usual Bochner spaces of $L^2/H^1$-functions mapping into a Hilbert space $\ZZ$.
This resembles the semi-discretization of a parabolic problem and we assume that $\VV$ is some kind of finite element space that satisfies the inverse inequality
\begin{align}\label{eq:assinv}
 \norm{v}{\VV}\lesssim h^{-2s} \norm{v}{\VV^\star}\quad\text{for all }v\in\VV
\end{align}
with some universal parameters $h,s>0$ (in the classical setting with $\AA=-\Delta$ and $\VV$ denoting an elementwise polynomial space, we have $s=1$ and $h$ being the mesh-size of $\VV$).

By $\TT$, we denote a partition of $[0,\tend]$ into compact intervals $T\in\TT$.
We define 
\begin{align*}
\XX_\TT&:=\set{v\in \XX}{v|_T\text{ is affine in }t\text{ for all }T\in\TT}.
%\YY_\TT&:=\set{v\in \YY}{v|_T\text{ is constant in }t\text{ for all }T\in\TT}.
\end{align*}
For brevity of presentation, we assume that $f\in L^2(0,\tend;\VV^\star)$ is piecewise affine on $\TT_0$.
Alternatively, we could introduce data oscillations and deal with general right-hand sides $f\in L^2(0,\tend;\VV^\star)$. 
We discretize~\eqref{eq:parabolic} by use of the Crank-Nicolson scheme, i.e., find $u_\TT\in\XX_\TT$ such that $u_\TT(0)=u_0$ and
\begin{align}\label{eq:implicitEuler}
 \frac{u_\TT(t_{i+1})-u_\TT(t_{i})}{|T_i|} +\AA u_\TT\big(\frac{t_{i+1}+t_i}{2}\big) = f\big(\frac{t_{i+1}+t_i}{2}\big),
\end{align}
for $T_i:=[t_i,t_{i+1}]\in\TT$ and all $i=1,\ldots, \#\TT$.
We use a standard residual error estimator of the form
\begin{align*}
 \eta_\TT^2=\sum_{T\in\TT} \eta_\TT(T)^2\quad\text{with}\quad\eta_\TT(T):=|T|\norm{\partial_t f -\partial_t\AA u_\TT}{L^2(T;\VV^\star)}.
\end{align*}
Mesh refinement is done by simple bisection of marked elements $T_i$ into descendants $T_{i,1}:=[t_i,(t_{i+1}+t_i)/2]$ and $T_{i,2}:=[(t_{i+1}+t_i)/2,t_{i+1}]$. We do not require any sort of mesh closure procedure or local quasi-uniformity of the meshes.

\subsection{Proof of the assumptions}
To verify the assumptions of Section~\ref{sec:abstract},
we embed the scheme into a Galerkin method. Note that this serves as a theoretical tool only as the Galerkin formulation below does
not need to be computed. We define the space
\begin{align*}
 \YY_\TT:=\set{v\in L^2(0,\tend;\VV)}{v|_T=\text{constant }\in\VV ,\, T\in\TT}\times \HH.
\end{align*}
This allows us to rewrite~\eqref{eq:implicitEuler} in the sense: Find $u_\TT\in\XX_\TT$ such that all $(v,\phi)\in\YY_\TT$ satisfy
\begin{align}\label{eq:discrete}
 a(u_\TT,(v,\phi)):=\int_0^\tend\dual{\partial_t u_\TT}{ v} + \dual{\AA u_\TT}{v}\,dt+\dual{u_\TT(0)}{\phi}= \int_0^\tend \dual{f}{v}\,dt+\dual{u_0}{\phi}.
\end{align}

\begin{lemma}\label{lem:equiv}
 The discretizations~\eqref{eq:implicitEuler} and~\eqref{eq:discrete} are equivalent.
\end{lemma}
\begin{proof}
The initial condition is satisfied exactly since we may choose $v=0$ and $\phi\in\VV$.
 Note that the functions $u_\TT$ and $f$ are affine on each element $T\in\TT$. Thus, for $w\in\{\AA u_\TT,f\}$, $v_0\in\VV$, and $T\in\TT$, there holds $\int_0^\tend\dual{w}{v_0\chi_T}\,dt =\dual{w(\frac{t_{i+1}+t_i}{2})}{v_0}|T|$, where $\chi_T\colon [0,\tend]\to \R$ is the indicator function of $T$. Since $\partial_t u$ is constant on each $T\in\TT$, we obtain analogously $\int_0^\tend\dual{\partial_t u}{v_0\chi_T}\,dt =\dual{\partial_t u|_T}{v_0}|T|$. With $\partial_t u|_{T_i} = (u(t_{i+1})-u(t_i))/|T_i|$, this concludes the proof.
\end{proof}

We define $\Pi_\TT(\cdot)$ by $(\Pi_\TT v)|_T:=|T|^{-1}\int_T v\,dt$ for all $T\in\TT$ as the orthogonal projection onto piecewise constants.
Note that $\Pi_\TT$ can be regarded as the orthogonal projection $\Pi_\TT\colon L^2(0,\tend;\ZZ)\to \set{v\in L^2(0,\tend;\ZZ)}{v|_T={\rm constant},T\in\TT}$ for any Hilbert space $\ZZ$. 

\begin{lemma}\label{lem:infsup}
The bilinear form $a(\cdot,\cdot)$ is bounded in the sense of~\eqref{eq:cont} with uniform constant $C_a$ and uniformly inf-sup stable~\eqref{eq:uniforminfsup}, where the constant $\gamma>0$ depends only on $\Omega$, $t_{\rm end}$, and {a positive lower} bound on
 $h^{2s}/\max_{T\in{\TT_0}}|T|$. 
\end{lemma}
\begin{remark}
 Note that while~\cite[Theorem~3.1]{lubich} shows uniform discrete stability of the Crank-Nicolson scheme, this does not imply inf-sup stability~\eqref{eq:uniforminfsup}. The reason is that the mentioned result proves a bound of the form $\norm{u_\TT}{\XX}\lesssim\norm{u_0}{\HH}+ \norm{f}{L^2(0,T;\VV^\star)}$, while discrete inf-sup stability~\eqref{eq:uniforminfsup} implies the stronger bound $\norm{u_\TT}{\XX}\lesssim \norm{u_0}{\HH}+\norm{\Pi_\TT f}{L^2(0,T;\VV^\star)}$. Hence, we require the CFL condition to prove~\eqref{eq:uniforminfsup}.
\end{remark}

\begin{proof}[Proof of Lemma~\ref{lem:infsup}]
\emph{{Step~1:}}
Continuous inf-sup stability of the weak form~\eqref{eq:weakparabolic} is well-known (see, e.g.,~\cite[Theorem~5.1]{infsup}) and so is boundedness of $a(\cdot,\cdot)$  (this follows immediately from the fact that $\XX$ embeds into $C^0(0,\tend;\HH)$, see, e.g.,~\cite[Ch. XVIII, \textsection 1, Th.1]{lions}).

\emph{{Step~2:}} For discrete inf-sup stability~\eqref{eq:uniforminfsup}, {we fix $v\in\XX_\TT$ and  construct $(w_0,\phi)\in\YY_\TT$ such that
\begin{align}\label{eq:iddd}
   C_0 a(v,(w_0,\phi))&\geq 
   \norm{\partial_t v}{L^2(0,\tend;\VV^\star)}^2 +\frac12\norm{\phi}{\HH}^2\quad\text{and}\quad 
   \norm{w_0}{L^2(0,\tend;\VV)}\leq C_0 \norm{\partial_t v}{L^2(0,\tend;\VV^\star)}
  \end{align}
for some constant $C_0>0$ that depends only on $\AA$.}
To that end, define $\phi=v(0)$ and $w_0 = \AA^{-T}\partial_t v$ and note $(w_0,\phi)\in\YY_\TT$. Since $\AA\colon \VV\to\VV^\star$ is coercive and bounded, also $\AA^{-T}\colon \VV^\star\to\VV$ is coercive and bounded. This implies
  \begin{align*}
  \begin{split}
   a(v,(w_0,\phi))&=\int_0^\tend \dual{\partial_t v}{\AA^{-T}\partial_t v} + \dual{\AA v}{\AA^{-T}\partial_t v}\,dt + \norm{\phi}{\HH}^2\\
   &\gtrsim 
   \int_0^\tend \norm{\partial_t v}{\VV^\star}^2 + \dual{v}{\partial_t v}\,dt + \norm{\phi}{\HH}^2\\
   &=
   \int_0^\tend \norm{\partial_t v}{\VV^\star}^2 + \frac12\partial_t\norm{v}{\HH}^2\,dt + \norm{\phi}{\HH}^2
   \\
   &= \norm{\partial_t v}{L^2(0,\tend;\VV^\star)}^2 +\frac12\big( \norm{v(\tend)}{\HH}^2-\norm{v(0)}{\HH}^2\big)+\norm{\phi}{\HH}^2\\
   &\geq 
   \norm{\partial_t v}{L^2(0,\tend;\VV^\star)}^2 +\frac12\norm{\phi}{\HH}^2
   \end{split}
  \end{align*}
 and $\norm{w_0}{L^2(0,\tend;\VV)}=\norm{\AA^{-T}\partial_t v}{L^2(0,\tend;\VV)}\simeq \norm{\partial_t v}{L^2(0,\tend;\VV^\star)}$ and thus shows~\eqref{eq:iddd}.
 
 \emph{{Step~3:}} {To bound the remaining part of $\norm{v}{\XX}$,}
 {we construct $(w_1,\phi)\in\YY_\TT$ with $\phi$ as in Step~2 such that
\begin{align}\label{eq:idddd} 
 \begin{split}
   a(v,(w_1,0)) &\geq C_1^{-1}
   \norm{\Pi_{\TT_0} v}{L^2(0,\tend;\VV)}^2-\frac{\norm{\phi}{\HH}^2}{2}\\
   &\qquad -\norm{\partial_t v}{L^2(0,\tend;\VV^\star)}\norm{(1-\Pi_{\TT_0})v}{L^2(0,\tend;\VV)},\\
   \norm{w_1}{L^2(0,\tend;\VV)}&\leq \norm{v}{L^2(0,\tend;\VV)}
   \end{split}
  \end{align}
for some constant $C_1>0$ that depends only on $\AA$.
To that end, let $w_1=\Pi_{\TT_0}v$ and observe}
  \begin{align*}
   \int_0^{\tend}\dual{\partial_t v}{w_1}\,dt &=  \int_0^{\tend}\dual{\partial_t v}{v}\,dt -\int_0^{\tend}\dual{\partial_t v}{(1-\Pi_{\TT_0})v}\,dt\\
   &= 
   \int_0^{\tend}\frac12 \partial_t\norm{ v}{\HH}^2\,dt-\int_0^\tend \dual{\partial_t v}{(1-\Pi_{\TT_0})v}\,dt  \\
   &\geq \frac12\big(\norm{v(\tend)}{\HH}^2-\norm{v(0)}{\HH}^2\big)-\norm{\partial_t v}{L^2(0,\tend;\VV^\star)}\norm{(1-\Pi_{\TT_0})v}{L^2(0,\tend;\VV)}\\
   &\geq -\frac12\norm{\phi}{\HH}^2-\norm{\partial_t v}{L^2(0,\tend;\VV^\star)}\norm{(1-\Pi_{\TT_0})v}{L^2(0,\tend;\VV)}.
  \end{align*}
Moreover, there holds (note that $\AA$ and $\Pi_\TT$ commute since $\AA$ is time-independent)
\begin{align*}
   \int_0^{\tend}\dual{\AA v}{w_1}\,dt =
    \int_0^{\tend}\dual{\AA \Pi_{\TT_0} v}{\Pi_{\TT_0}v}\,dt \simeq \norm{\Pi_{\TT_0} v}{L^2(0,\tend;\VV)}^2.
\end{align*}
The combination of the two previous estimates shows
 \begin{align*}
 \begin{split}
   a(v,(w_1,0)) &= \int_0^{\tend}\dual{\partial_t v}{w_1}\,dt+ \int_0^{\tend}\dual{\AA v}{w_1}\,dt\\
   &\geq C_1^{-1}
   \norm{\Pi_{\TT_0} v}{L^2(0,\tend;\VV)}^2-\frac{\norm{\phi}{\HH}^2}{2}-\norm{\partial_t v}{L^2(0,\tend;\VV^\star)}\norm{(1-\Pi_{\TT_0})v}{L^2(0,\tend;\VV)},
   \end{split}
  \end{align*}
  where $C_1>0$ depends only on $\AA$. Moreover, there holds $\norm{w_1}{L^2(0,\tend;\VV)}\leq  \norm{v}{L^2(0,\tend;\VV)}$ and hence we prove~\eqref{eq:idddd}.
  
\emph{{Step~4:}} {To conclude the proof, we require a bound of the form $\norm{(1-\Pi_{\TT_0})v}{L^2(0,\tend;\VV)}\lesssim \norm{\partial_t v}{L^2(0,\tend;\VV^\star)}$.} To that end, we use the inverse inequality~\eqref{eq:assinv} together with the standard approximation properties of the $L^2$-orthogonal projection onto
piecewise constants (see e.g.~\cite{l2ortho} for the elementary result for vector valued functions) to estimate 
\begin{align}\label{eq:approxest}
 \frac{h^{2s}}{\max_{T\in{\TT_0}}|T|}\norm{(1-\Pi_{\TT_0}) v}{L^2(0,\tend;\VV)}\lesssim \frac{\norm{(1-\Pi_{\TT_0}) v}{L^2(0,\tend;\VV^\star)}}{\max_{T\in{\TT_0}}|T|} \lesssim \norm{\partial_t v}{L^2(0,\tend;\VV^\star)}.
\end{align}
This, together with~\eqref{eq:idddd} shows
\begin{align}\label{eq:idd}
 a(v,(w_1,0))\geq C_1^{-1}
   \norm{\Pi_{\TT_0} v}{L^2(0,\tend;\VV)}^2-\frac{\norm{\phi}{\HH}^2}{2}-C\norm{\partial_t v}{L^2(0,\tend;\VV^\star)}^2,
\end{align}
where $C>0$ depends only on the inverse inequality~\eqref{eq:assinv} and {a positive lower bound for $h^{2s}/\max_{T\in{\TT_0}}|T|$}.
By defining $w:= w_0+\alpha w_1$ with $\alpha:=\min\{{C_0^{-1}},\frac{{C_0^{-1}}}{2C}\}$, the combination of~\eqref{eq:iddd} and~\eqref{eq:idd} shows
\begin{align*}
 a(v,(w,\phi))&\geq C_0^{-1}\norm{\partial_t v}{L^2(0,\tend;\VV^\star)}^2 +\frac12(C_0^{-1}-\alpha)\norm{\phi}{\HH}^2\\
 &\qquad + C_1^{-1}\alpha
   \norm{\Pi_{\TT_0} v}{L^2(0,\tend;\VV)}^2-C\alpha\norm{\partial_t v}{L^2(0,\tend;\VV^\star)}^2
 \\
 &\geq \frac{C_0^{-1}}2\norm{\partial_t v}{L^2(0,\tend;\VV^\star)}^2 +C_1^{-1}\alpha
   \norm{\Pi_{\TT_0} v}{L^2(0,\tend;\VV)}^2.
\end{align*}
With~\eqref{eq:approxest}, we get 
\begin{align*}
\frac12\norm{v}{L^2(0,\tend;\VV)}^2&\leq \norm{(1-\Pi_{\TT_0})v}{L^2(0,\tend;\VV)}^2+\norm{\Pi_{\TT_0}v}{L^2(0,\tend;\VV)}^2\\
&\leq C'\norm{\partial_t v}{L^2(0,\tend;\VV^\star)}^2+ \norm{\Pi_{\TT_0}v}{L^2(0,\tend;\VV)}^2,
\end{align*}
where $C'>0$ depends only on the constant in~\eqref{eq:assinv} and {a positive lower bound for $h^{2s}/\max_{T\in{\TT_0}}|T|$}.
The combination of the last two estimates concludes 
\begin{align*}
\norm{v}{\XX}^2&=\norm{v}{L^2(0,\tend;\VV)}^2 + \norm{\partial_tv}{L^2(0,\tend;\VV^\star)}^2 \\
&\leq
(2C'+1)\norm{\partial_t v}{L^2(0,\tend;\VV^\star)}^2+ 2\norm{\Pi_{\TT_0}v}{L^2(0,\tend;\VV)}^2\\
&\leq 
(2C'+2)\Big(\norm{\partial_t v}{L^2(0,\tend;\VV^\star)}^2+ \norm{\Pi_{\TT_0}v}{L^2(0,\tend;\VV)}^2\Big)\\
&\leq (2C'+2)\max\{2C_0,C_1/\alpha\}
\,a(v,(w,\phi)).
\end{align*}
Moreover, we have $\norm{w}{L^2(0,\tend;\VV)}\lesssim \norm{v}{\XX}$ and hence $\norm{(w,\phi)}{\YY}\lesssim \norm{v}{\XX}$, which concludes the proof.
%Since the mapping $v\mapsto (w,\phi)$ is linear and $v\in\XX_\TT$ implies $(w,\phi)\in \YY_\TT$ for all $\TT\in\T$, we set $Tv:=(w,\phi)$ and conclude the proof.
\end{proof}

The residual error estimator $\eta_\TT$ is very simple and the proofs of this section follow well known paths.
\begin{lemma}\label{lem:rel}
Given a mesh $\TT\in \T$, there holds reliability
\begin{align*}
 \norm{u-u_\TT}{\XX}\leq C_{\rm rel}\eta_\TT.
\end{align*}
Given a refinement $\widehat\TT$ of $\TT$, there holds
\begin{align*}
 \norm{u_{\widehat\TT}-u_\TT}{\XX}\leq C_{\rm drel}\Big(\sum_{T\in\TT\setminus\widehat\TT}\eta_\TT(T)^2\Big)^{1/2},
\end{align*}
where $C_{\rm rel},C_{\rm drel}>0$ depend only on $\Omega$, $\tend$ and  $C_{\rm drel}$ depends additionally on a positive lower bound for
 $h^{2s}/\max_{T\in{\TT_0}}|T|$.
\end{lemma}
\begin{proof}
Let $(\widehat v,\phi)\in\YY_{\widehat\TT}$ and $(v,\phi)\in\YY_\TT$ with $\phi\in\HH$. There holds with~\eqref{eq:discrete}
\begin{align*}
 a(u_{\widehat\TT}-u_{\TT},(\widehat v,\phi)) &= a(u_{\widehat\TT}-u_{\TT},(\widehat v-v,0))
 =
 \int_0^\tend\dual{f-\partial_t u_\TT - \AA u_\TT }{\widehat v-v}\,dt.
\end{align*}
We may choose $v=\widehat v$ on all $T\in\TT\cap\widehat\TT$ and zero elsewhere. This implies
\begin{align*}
  a(u_{\widehat\TT}-u_{\TT},(\widehat v,\phi)) \leq
  \Big(\sum_{T\in\TT\setminus\widehat\TT} \norm{f-\partial_t u_\TT - \AA u_\TT }{L^2(T;\VV^\star)}^2\Big)^{1/2}\norm{\widehat v}{L^2(0,\tend,\VV)}.
\end{align*}
By definition of $u_\TT$ in~\eqref{eq:implicitEuler}, there holds $\Pi_\TT(f-\partial_t u_\TT - \AA u_\TT)|_{T_i}=(f-\partial_t u_\TT - \AA u_\TT)((t_{i+1}+t_i)/2)=0$ for all $i=1,\ldots,\#\TT$.
Hence, we have for all $T\in\TT$
\begin{align*}
  \norm{f-\partial_t u_\TT - \AA u_\TT }{L^2(T;\VV^\star)}
  \lesssim |T| \norm{\partial_t f - \partial_t\AA u_\TT }{L^2(T;\VV^\star)}.
\end{align*}
Taking the supremum over all $\widehat v$ and Lemma~\ref{lem:infsup} conclude the proof of discrete reliability. The proof of reliability follows by exactly the same arguments when we replace $u_{\widehat\TT}$ with $u$ (and may use the continuous inf-sup stability instead of the discrete one, thus avoiding the mesh condition).
\end{proof}

\begin{lemma}\label{lem:estred}
Let $\TT\in\T$ and let $\widehat\TT$ be a refinement of $\TT$. Then, the error estimator satisfies
\begin{align*}
 \sum_{T\in\widehat\TT\setminus\TT}\eta_{\widehat\TT}(T)^2\leq q\sum_{T\in\TT\setminus\widehat\TT} \eta_{\TT}(T)^2 + C \norm{u_{\widehat\TT}-u_{\TT}}{L^2(0,\tend;\VV)}^2
\end{align*}
as well as
\begin{align*}
  \Big|\Big(\sum_{T\in\widehat\TT\cap\TT}\eta_{\widehat\TT}(T)^2\Big)^{1/2}-\Big(\sum_{T\in\TT\cap\widehat\TT}\eta_{\TT}(T)^2\Big)^{1/2}\Big|\leq  C \norm{u_{\widehat\TT}-u_{\TT}}{L^2(0,\tend;\VV)},
\end{align*}
where the constants $1/4<q <1$ and $C>0$ depend only on $\AA$.
\end{lemma}
\begin{proof}
 Let $T\in\TT\setminus\widehat\TT$ and let $T_1,T_2\in\widehat\TT$ with $T_1\cup T_2=T$ and $|T_1|=|T_2|=|T|/2$. There holds for all $\delta>0$ that
 \begin{align*}
  \eta_{\widehat\TT}(T)^2 &= |T_1|^2\norm{\partial_t f -\partial_t\AA u_{\widehat\TT}}{L^2(T;\VV^\star)}^2\\
  &\leq \frac14 |T|^2(1+\delta)\norm{\partial_t f -\partial_t\AA u_{\TT}}{L^2(T;\VV^\star)}^2+
  (1+\delta^{-1})|T_1|^2\norm{ \partial_t\AA (u_{\widehat\TT}-u_{\TT})}{L^2(T;\VV^\star)}^2.
 \end{align*}
Since $\AA (u_{\widehat\TT}-u_{\TT})$ is piecewise linear in $\widehat\TT$, there holds
\begin{align*}
 |T_1|\norm{\partial_t\AA (u_{\widehat\TT}-u_{\TT})}{L^2(T;\VV^\star)}\lesssim
 \norm{\AA (u_{\widehat\TT}-u_{\TT})}{L^2(T;\VV^\star)}\simeq \norm{u_{\widehat\TT}-u_{\TT}}{L^2(T;\VV)}.
\end{align*}
This shows
\begin{align*}
  \eta_{\widehat\TT}(T)^2 \leq (1+\delta)\frac{1}{4} \eta_\TT^2(T) + C_\delta \norm{u_{\widehat\TT}-u_{\TT}}{L^2(T;\VV)}^2.
\end{align*}
Summing up over all $T\in\widehat\TT\setminus\TT$ concludes the proof of the first statement. The second statement follows analogously.
\end{proof}

\begin{theorem}\label{thm:timeopt}
 Algorithm~\ref{alg:adaptive} for the time discretization of the parabolic problem~\eqref{eq:parabolic} is optimal~\eqref{eq:opt} {for all $0<\theta<\theta_\star$}.
\end{theorem}
\begin{proof}
Note that the 1D bisection we use as a mesh refinement is technically newest-vertex-bisection for intervals and hence fits
into our abstract framework (see also~\cite[Sections~2.4--2.5]{axioms} for details).
{Note that $h^{2s}/\max_{T\in{\TT_0}}|T|>0$ due to $h>0$ and $\tend<\infty$.
Thus, Lemma~\ref{lem:infsup} shows~\eqref{eq:uniforminfsup}.}
Lemmas~\ref{lem:rel}--\ref{lem:estred} show (A1), (A2), and~(A4). Hence, the problem fits into the abstract framework of Section~\ref{sec:abstract} and Theorem~\ref{thm:opt} implies the statement. 
\end{proof}
\begin{remark}
 Obviously, the CFL condition in the theorem above is not optimal and is an artifact of the proof. If one replaces the Crank-Nicolson scheme with a $\theta$-scheme for $\theta>1/2$, one could remove the CFL condition, but has to construct nested test spaces $\YY_\TT$ that emulate the time-stepping scheme. Currently, we do not know how to do this. As the numerical experiment in the next section shows, the violation of the CFL condition does not seem to influence the convergence behavior.
%  Moreover, if the Aubin-Nitsche exploit is available, there holds $\norm{u-u_\TT}{L^2(0,\tend;L^2(\Omega))}\lesssim h\norm{u-u_\TT}{\XX}$ and analysis of the generic startup singularity reveals that $h^2\ll \max_{T\in\TT}|T|^{1/2}$ is sufficient for convergence
\end{remark}
\subsection{Numerical experiment}
While the applications in Sections~\ref{sec:stokes}--\ref{sec:fembem} are explored quite extensively by means of numerical experiments (see, e.g.,~\cite{fembem}), we are not aware of a numerical test of the adaptive time-stepping algorithm proposed in this section. Therefore, we consider the heat equation on $\Omega=[0,1]^2$ with $\tend=1$, i.e.,
\begin{align*}
 \partial_t u -\Delta u = 0\quad\text{and}\quad u|_{\partial\Omega} =0
\end{align*}
with initial condition  $u(0)=1$. We semi-discretize the problem with a lowest order finite element method on a uniform grid with $\approx 2\cdot 10^3$ elements (using the Matlab FEM package P1AFEM~\cite{p1afem}). Here $\VV=\SS^1_0(\TT)$ is the lowest order finite element space equipped with the $H^1(\Omega)$-norm and $\HH$ is equipped with the $L^2(\Omega)$-norm. We project the non-matching initial condition onto the finite element space {with respect to the scalar product $\dual{\cdot}{\cdot}+\frac{|T_0|}{2}\dual{\AA\cdot}{\cdot}$. (This is the natural projection considering that we aim to minimize the error in the initial condition w.r.t. $\norm{\cdot}{L^2(D)}$ as well as in the first time step w.r.t. $\sqrt{|T_0|}\norm{\nabla\cdot}{L^2(D)}$)} and confirm that the operator $\AA:=-\Delta\colon \VV\to \VV^\star$ is bounded and elliptic. We run Algorithm~\ref{alg:adaptive} and plot the error as well as the estimator in Figure~\ref{fig:timeconv}. Since the exact solution is unknown, we compare the current approximation with the approximation on the finest grid. We observe convergence with rate $\mathcal{O}(\#\TT^{-1})$. Note that this is the expected rate as $\partial_t u$ appears in the norm of $\XX$ and is approximated by the piecewise linear functions in $\XX_\TT$ with rate at most one. The slight super convergence of the adaptive error is due to the fact that we compare with the solution on the finest grid instead of computing the (unknown) exact error. The reduced rate of the uniform mesh refinement is due to the startup singularity for non-matching boundary conditions (note that due to the finite spatial resolution, this is not a real singularity as can be observed by the increased rate of convergence in the final few steps of the uniform algorithm). Indeed, standard regularity results for parabolic PDEs show (see, e.g.,~\cite[Theorem 7.5]{evansPDE}) that for $u_0\in H^1_0(\Omega)\cap H^2(\Omega)$ and smooth right-hand side there holds $\partial_t u\in L^2(0,\tend;H^1_0(\Omega))$. The a~priori regularity for $u_0\in L^2(\Omega)$ states $u\in L^2(0,\tend;H^1_0(\Omega))$. Interpolation of those two estimates together with the fact  $u_0=1\in H_0^{1/2-\delta}(\Omega)$ leads to $\partial_t u\in H^{1/4-\delta}(0,\tend;H^1_0(\Omega))$ for all $\delta>0$ and explains the observed rate $\mathcal{O}(\#\TT^{-1/4})$ in Figure~\ref{fig:timeconv}.

\begin{figure}
\psfrag{adaptiverror}{\tiny adaptive (error)}
\psfrag{adaptiveest}{\tiny adaptive (estimator)}
\psfrag{adaptiveadaptive}{\tiny adaptive}
\psfrag{uniform}{\tiny uniform}
\psfrag{uniferror}{\tiny uniform (error)}
\psfrag{unifest}{\tiny uniform (estimator)}
\psfrag{cost}[cc][cc]{\tiny number of time steps $\#\TT$}
\psfrag{time}[cc][cc]{\tiny $t\in[0,1]$}
\psfrag{size}[cc][cc]{\tiny size of time step}
 \includegraphics[width=0.47\textwidth]{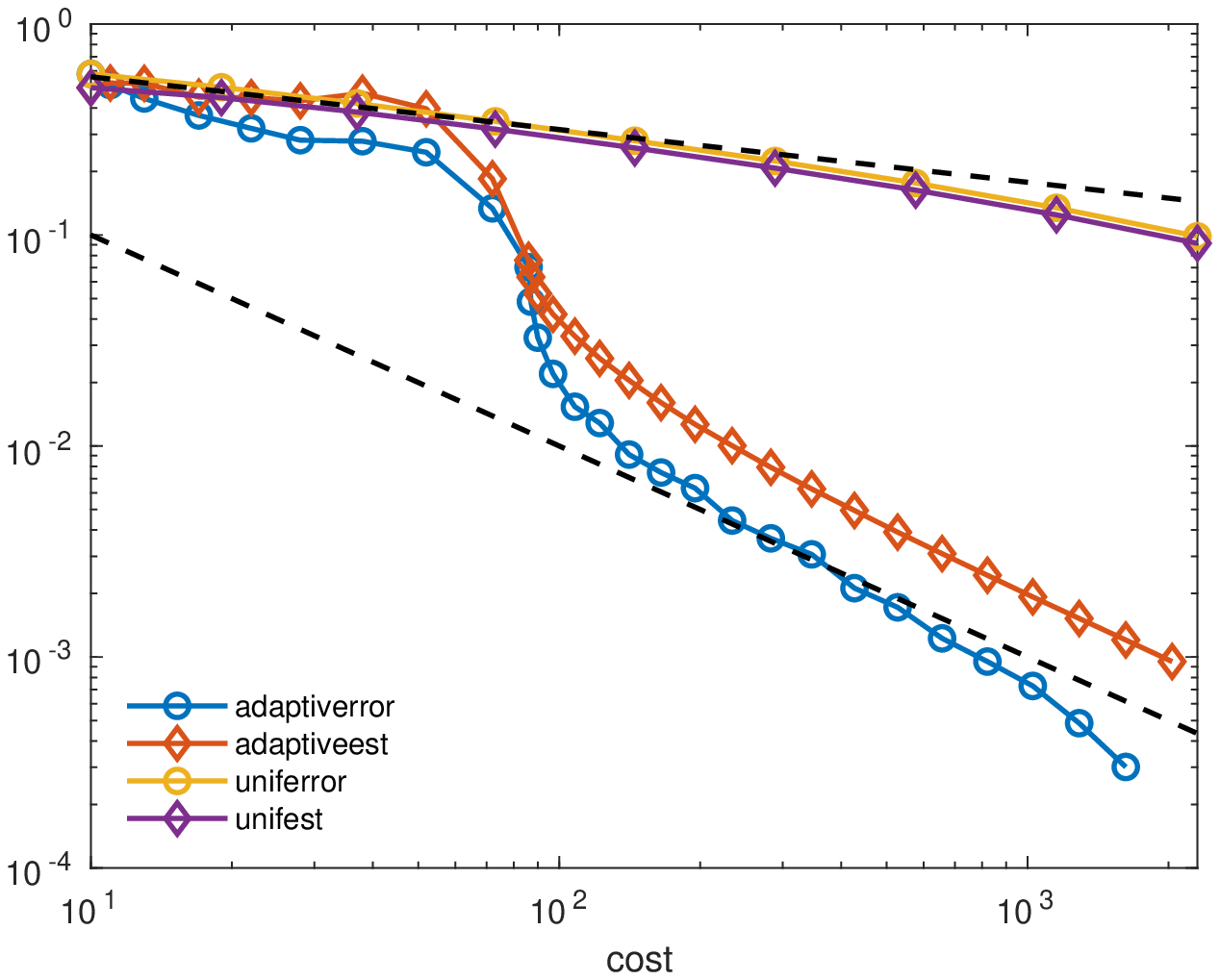}\hspace{7mm}%
 \includegraphics[width=0.47\textwidth]{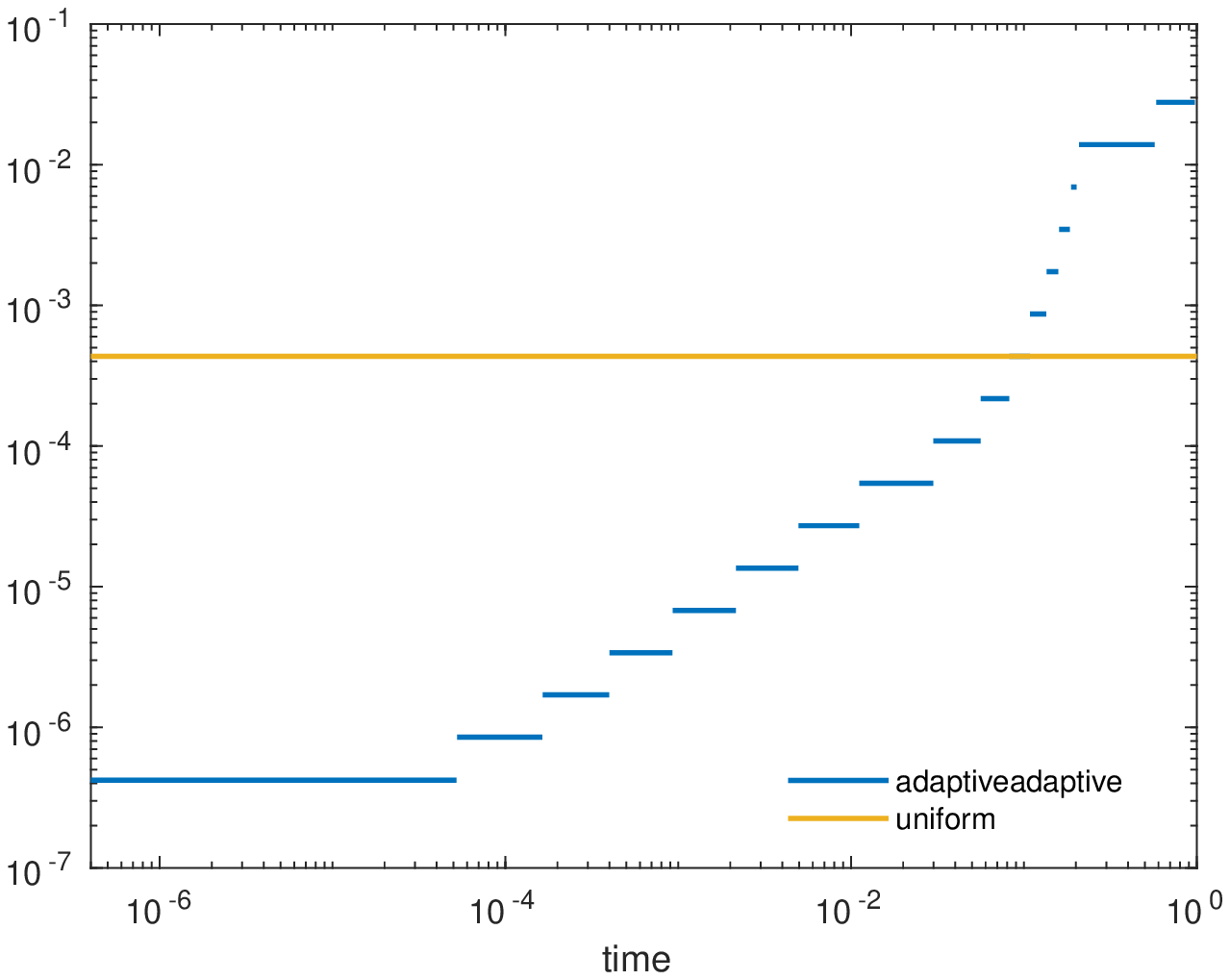}  
 \caption{(left plot) Convergence of error $\norm{u-u_\TT}{\XX}$ (comparison with finest approximation) and estimator $\eta_\TT$ for adaptive ($\theta=1/2$) and uniform mesh refinement. The dashed lines represent $\mathcal{O}(\#\TT^{-1/4})$ for uniform refinement and $\mathcal{O}(\#\TT^{-1})$ for adaptive refinement. (right plot) Sizes of local time steps of the last iteration of the adaptive/uniform algorithm plotted over their position in the time interval $[0,1]$.}
 \label{fig:timeconv}
 \end{figure}

\bibliographystyle{plain}
\bibliography{literature}
\end{document}